\documentclass{article}[11pt]
\usepackage[ruled,vlined,linesnumbered]{algorithm2e}
\usepackage{graphicx}
\usepackage[thinlines]{easytable}
\usepackage{enumerate}
\usepackage{fullpage}
\usepackage{amsmath,amsfonts,amssymb,latexsym}
\usepackage{setspace}
\usepackage{color,cite}
\usepackage[cp1250]{inputenc}
\usepackage{multirow}
\usepackage{hhline}

\newtheorem{theorem}{Theorem}[section]
\newtheorem{definition}{Definition}[section]
\newtheorem{proposition}[theorem]{Proposition}

\newtheorem{corollary}[theorem]{Corollary}
\newtheorem{remark}[theorem]{Remark}

\newcommand{\R}{{\mathbb R}}
\newcommand{\wt}{\,{\mbox{\tt:}\:}}

\newcommand{\qed}{\hfill ~$\square$\bigskip}

\title{On Equistable, Split, CIS, and Related Classes of Graphs}

\author{
Endre Boros and Vladimir Gurvich\\
\small MSIS Department and RUTCOR, Rutgers University, New Jersey, USA\\
\small 100 Rockafeller Rd, Piscataway NJ 08854, USA\\
\small \texttt{\{Endre.Boros,Vladimir.Gurvich\}@rutgers.edu}\\
\and
Martin Milani\v c\\
\small University of Primorska, UP IAM, Muzejski trg 2, SI6000 Koper, Slovenia\\
\small University of Primorska, UP FAMNIT, Glagolja\v ska 8, SI6000 Koper, Slovenia\\
\small \texttt{martin.milanic@upr.si}
}

\date{}

\begin{document}

\maketitle

\begin{abstract}
We consider several graphs classes defined in terms of conditions on cliques and stable sets, including
CIS, split, equistable, and other related classes.
We pursue a systematic study of the relations between them.
As part of this study, we introduce two generalizations of CIS graphs, obtain a new characterization of split graphs, and
a characterization of CIS line graphs.
\end{abstract}



\section{Introduction}

{ Many graph classes can be defined by imposing conditions that must be satisfied by some or all cliques
and/or by some or all stable sets of the graph. Examples of such conditions include:
the existence of a partition of the graph's vertex set
into a specified number of cliques and stable sets~\cite{FH77,Br96,EkimGimbel},
the existence of a stable set the deletion of which results in a graph satisfying certain properties~\cite{BLSz,HC07, BHLL,YY,BDLS,BBKNP},
the existence of a linear weight function on the vertices separating the stable sets
(or the maximal stable sets) from all other vertex subsets~\cite{CH77,MP,Pay},
restrictions on the relations between the sizes of maximal stable sets~\cite{Plummer},
and restrictions related to the intersections between cliques
and stable sets~\cite{Hoang,ABG06,Berge}.}

In particular, a graph is said to be {\it CIS} if every maximal clique $C$ and every maximal stable set $S$ in $G$ intersect, that is, $C\cap S\neq\emptyset$,
and {\it almost CIS} if this condition holds, except for a unique pair.
CIS graphs were studied in a series of papers~\cite{ABG06,ABG10,BGZ09,BGM13,Gur09}; see also~\cite{Zang} for some earlier references.
A graph $G = (V,E)$ is said to be {\em split} if it admits a {\em split partition},
that is, a pair $(C,S)$ such that $C$ is a clique in $G$, $S$ is a stable set in $G$, $C\cup S = V$ and $C\cap S = \emptyset$.
Split graphs were introduced by F\"oldes and Hammer in~\cite{FH77}, where several characterizations were also given.
The almost CIS graphs conicide with the split graphs with a unique split partition.
This was conjectured by Boros et al.~\cite{BGZ09} and proved by Wu et al.~\cite{WZZ}.

Another related family is formed by the {\it semi-weakly CIS graphs} (see Section~\ref{sec:CIS} for definition).
Semi-weakly CIS graphs appeared in the literature under the name of {\it $(2,\ell)$-CIS graphs}
(for large enough $\ell$)~\cite{Gur11} and are equivalent to {\it general partition graphs}, which arose in the geometric setting of
lattice polygon triangulations~\cite{DTR} and were studied in a series of papers~\cite{ADMR,DTDRM,DTHR, DTR0,DRH,KLLM}.
Semi-weakly CIS graphs are strongly related to exact anti-blockers (see~\cite{Gur11} for definitions and proofs.)

In $1980$, Payan~\cite{Pay} introduced {\it equistable graphs} as a generalization of the well known class of
threshold graphs~\cite{MP}. (See Section~\ref{sec:equistable} for definitions.) In $2009$, Orlin proved that every general
partition graph is equistable and conjectured that the converse holds as well; see~\cite{MM-2011}.
{The conjecture was recently disproved~\cite{MTT}.}

In this paper, we consider CIS, almost CIS, split, semi-weakly CIS, equistable
graphs, as well as several further related graph classes (including
edge simplicial, triangle, normal, and perfect graphs) and pursue a systematic study of
relations between them. The known inclusion relations between $27$ graph families considered in this paper
are summarized in the Hasse diagram on p.~\pageref{fig:Hasse}.
We introduce two new graph classes, the class of {\it weakly CIS graphs}---a generalization of CIS graphs---, and
the class of {\it weakly triangle graphs}---a generalization of triangle graphs.
Focusing mostly on self-complementary graph properties, we show that every
edge simplicial co-edge simplicial graph is split, which leads to a new characterization
of split graphs (Theorem~\ref{thm:split-characterization}).
We show the existence of split edge simplicial co-edge simplicial non-CIS graphs, giving a construction based on random graphs,
and an explicit infinite family based on finite projective planes.
Finally, we examine some of the inclusion relations in the case of line graphs, giving several equivalent
characterizations of CIS line graphs, and a polynomial time algorithm for their recognition.

The paper is structured as follows.
In Section~\ref{sec:equistable}, we briefly summarize the definitions and results related to
equistable graphs.
In Section~\ref{sec:further}, we extend the family of classes studied in Section~\ref{sec:equistable},
discussing the edge simplicial, triangle and weakly triangle graphs.
In Section~\ref{sec:CIS}, we discuss two generalizations of CIS graphs.
In Section~\ref{sec:split}, we examine relationships between split, CIS, almost CIS and edge simplicial graphs.
In Section~\ref{sec:split-constructions}, we show the existence
of split edge simplicial co-edge simplicial non-CIS graphs. In Section~\ref{sec:relations}, we summarize and justify the known
inclusion relations between the $27$ considered graph properties. We characterize CIS line graphs in Section~\ref{sec:line}, and
conclude the paper with some open questions in Section~\ref{sec:open}.

\bigskip

\noindent{\bf Notation and preliminary definitions.}
Given a graph $G$ and a vertex $v\in V(G)$, we write $N_G(v)$ (or just $N(v)$ if the graph is clear) for the
{\it neighborhood} of $v$, that is,
the set of vertices adjacent to $v$. By $N[v]$ we denote the {\it closed neighborhood} of $v$, that is, the set $N(v)\cup \{v\}$.
For a graph $G$ and $X\subseteq V(G)$, we write $N(X)$ for $(\cup_{v\in X}N(v))\setminus X$.
Given a graph $G$,  a {\it clique} in $G$ is a set of pairwise
adjacent vertices, and {\it stable} (or {\it independent}) set  in $G$ is a set of pairwise non-adjacent vertices, and
For a subset of vertices $X\subseteq V(G)$, we will denote by $G[X]$ the subgraph of $G$ induced by $X$.
By $P_n$, $C_n$, and $K_n$, we denote the path, the cycle, and the complete graph on $n$ vertices, respectively.
By $K_{m,n}$ we denote the complete bipartite graph with parts of size $m$ and $n$.
For two vertex-disjoint graphs $G_1$ and $G_2$, the {\em disjoint union} of $G_1$ and $G_2$ is the graph defined by
$(V(G_1)\cup V(G_2),E(G_1)\cup E(G_2))$ and denoted by $G_1+G_2$.
The disjoint union of $k$ graphs isomorphic to a graph $H$ will be denoted by $kH$.
The {\it join} of $G_1$ and $G_2$ is the graph defined by
$(V(G_1)\cup V(G_2),E(G_1)\cup E(G_2)\cup\{x_1x_2\mid x_1\in V(G_1),~x_2\in V(G_2)\})$ and denoted by $G_1*G_2$.
The {\it complement} of a graph $G = (V,E)$ is the graph $\overline{G}$ with the same vertex set as $G$,
in which two distinct vertices are adjacent if and only if they are non-adjacent in $G$. A graph $G$ is said to be {\it self-complementary} if it is isomorphic to its complement.
For a set $M$ of graphs, we say that a graph $G$ is {\it $M$-free} if no member of $M$ is an induced subgraph of $G$. For a graph $H$, we say that $G$ is $H$-free if it is $\{H\}$-free.
We say that a graph $G$ is a {\em line graph} if there exists a graph $H$ such that there is a bijective mapping
$\varphi: V(G)\to E(H)$ such that two distinct vertices $u,v\in V(G)$ are adjacent in $G$ if and only if the
edges $\varphi(u)$ and $\varphi(v)$ of $H$ have a common endpoint.
If this is the case, then graph $H$ is said to be a {\em root graph} of $G$, while $G=L(H)$ is the {\em line graph} of $H$.
Except for $G= K_3$, the root graph of a connected line graph $G$ is unique (up to isomorphism), and can be computed in linear time~\cite{Roussopoulos}.
For graph classes not defined here, we refer to~\cite{BLS}.

\section{Equistable graphs, strongly equistable graphs, and general partition graphs}\label{sec:equistable}

In $1980$, Payan introduced equistable graphs as a generalization of the well known class of threshold graphs~\cite{MP}:
A graph $G=(V,E)$ is called {\it equistable} if and only if there exists a
mapping $\varphi :V\to \R_+$
such that for all $S\subseteq V$, $S \textrm{\ is\ a\ maximal\
     stable\ set\ of\ } G$ if and only if \hbox{$\varphi(S):= \sum_{v\in S}\varphi(v) = 1\,.$}~\cite{Pay}.
The mapping $\varphi$ is called an {\it equistable weight function} of $G$.
In~\cite{MPS}, Mahadev et al.~introduced a subclass of equistable graphs, the so-called strongly equistable graphs.
For a graph $G$, we denote by ${\cal S}(G)$ the set of all maximal stable sets of $G$, and by ${\cal T}(G)$ the set of all other nonempty subsets of $V(G)$. A graph $G=(V,E)$
is said to be {\it strongly equistable} if for each $T\in {\cal T}(G)$ and each $\gamma\le 1$ there exists a weight function
$\varphi:V\to\R_+$ such that  $\varphi(S) = 1{\rm\ for\ all\ } S\in {\cal S}(G)$, and $\varphi(T)\neq \gamma$.
Mahadev et al.~showed that every strongly equistable graph is equistable, and conjectured that the converse assertion is valid.

\begin{theorem}[Mahadev-Peled-Sun~\cite{MPS}]\label{thm:Seq-eq}
All strongly equistable graphs are equistable.
\end{theorem}


Furthermore, Mahadev, Peled and Sun \cite{MPS} conjectured that equistable graphs are strongly equistable. This statement holds for a class of graphs containing all perfect graphs~\cite{MPS}, for series-parallel graphs~\cite{KP}, for AT-free graphs~\cite{MM-2011}, for various product graphs~\cite{MM-2011}, for line graphs~\cite{LM}, for simplicial graphs~\cite{LM}, for very well covered graphs~\cite{LM}, and for EPT graphs~\cite{AGKM13}. { Recently, however, it was disproved~\cite{MTT}; a $9$-vertex, $14$-edge triangle-free graph was constructed such that the complement of its line graph is equistable but not strongly equistable.}

A graph $G= (V,E)$ is a {\it general partition graph} if there exists a set $U$ and an assignment of non-empty subsets $U_x\subseteq U$
to the vertices of $G$ such that two vertices $x$ and $y$ are adjacent if and only if $U_x\cap U_y\neq\emptyset$, and
for every maximal stable set $S$ of $G$, the set $\{U_x\,:\,x\in S\}$ is a partition of $U$.

\begin{theorem}[see~\cite{MM-2011}]\label{thm:gpg-seq}
All general partition graphs are strongly equistable.
\end{theorem}


In 2009 Orlin (personal communication, 2009, see~\cite{MM-2011}) conjectured that equistable graphs are general partition.
This conjecture, if true, would imply that the classes of general partition graphs, strongly equistable graphs and equistable graphs all coincide~\cite{MM-2011}, and would provide a combinatorial characterization of equistable graphs. It follows from the results by Korach et al.~\cite{KPR} and Peled and Rotics~\cite{PR} that Orlin's conjecture holds within the class of distance-hereditary graphs and chordal graphs, respectively. It has also been verified for AT-free graphs, various product graphs~\cite{MM-2011}, for line graphs~\cite{LM}, for simplicial graphs~\cite{LM}, for very well covered graphs~\cite{LM}, and for EPT graphs~\cite{AGKM13}.
{ Recently, however, it was shown that Orlin's conjecture fails for
complements of line graphs of triangle-free graphs~\cite{MTT}. }

\section{Further related classes: edge simplicial, triangle and weakly triangle graphs}\label{sec:further}

The above chain of inclusion relations
$$\emph{general partition graphs}\subset \emph{strongly equistable graphs}\subset \emph{equistable graphs}$$
can be extended further. Following~\cite{OBDFG,OZ}, we say that a graph is a {\it triangle graph} if it satisfies the following.

\noindent
{\bf Triangle Condition}. For every maximal stable set $S$ in $G= (V,E)$ and every edge $uv$
in $G - S$ there is a vertex $s\in S$ such that $\{u,v,s\}$ induces a triangle in $G$.

The triangle condition was introduced by McAvaney et al.~in~\cite{MRDT}, who proved
that all general partition graphs satisfy the condition.
In~\cite{OBDFG} and~\cite{OZ}, Orlovich et al.~and Orlovich and Zverovich, respectively, discuss several complexity results for triangle graphs.
In the equistable graphs literature, a condition equivalent to the triangle condition was also used, expressed in terms of induced $4$-vertex paths
and maximal stable sets. We say that an induced $4$-vertex path $P_4(a,b,c,d)$ in a graph $G$ is {\it bad} if there exists
a maximal stable set $S$ in $G$ containing $a$ and $d$ such that no vertex from $S$ is adjacent both to $b$ and $c$.
The following two results are equivalent.
\begin{theorem}[Mahadev-Peled-Sun~\cite{MPS}]
\label{thm:badP4}
If $G$ is equistable, then $G$ contains no bad~$P_4$.
\end{theorem}

\begin{theorem}[Miklavi\v c-Milani\v c~\cite{MM-2011}] \label{thm:triangle}
Every equistable graph is a triangle graph.
\end{theorem}

The following closure property of triangle graphs will be used in later sections.

\begin{proposition}\label{prop:triangle-disjoint-union-join}
If $G_1$ and $G_2$ are triangle, then their disjoint union $G_1+G_2$ and join $G_1*G_2$ are also triangle.
\end{proposition}

\begin{proof}
It can be seen that the collections of maximal stable sets of $G_1$, $G_2$, $G_1+G_2$, and $G_1*G_2$, denoted respectively by
${\cal S}(G_1)$, ${\cal S}(G_2)$, ${\cal S}(G_1+G_2)$, ${\cal S}(G_1*G_2)$, are related as follows:
\begin{eqnarray}
\label{eq1}
  {\cal S}(G_1+G_2) &=& \{S_1\cup S_2\mid S_1\in {\cal S}(G_1)\,, S_2\in {\cal S}(G_2)\}\,,\\
\label{eq2}
{\cal S}(G_1*G_2) &=& {\cal S}(G_1)\cup {\cal S}(G_2) \text{ (disjoint union)}\,.
\end{eqnarray}
Suppose that $G_1$ and $G_2$ are triangle.
Let $S$ be a maximal stable set in $G:=G_1+G_2$, and let $uv$ be an edge in
$G - S$. Without loss of generality, we may assume that $u,v\in V(G_1)$.
By~\eqref{eq1}, we have $S = S_1\cup S_2$ for some $S_1\in {\cal S}(G_1)\,, S_2\in {\cal S}(G_2)$, and
the triangle condition for $G_1$
implies the existence of a vertex $s\in S_1\subseteq S$ such that $\{u,v,s\}$ induces a triangle in $G_1$ (and hence in $G$). Therefore, $G$ is triangle.

Now, let $S$ be a maximal stable set in $G':=G_1*G_2$, and let $uv$ be an edge in
$G' - S$. By~\eqref{eq2}, we may assume without loss of generality that $S$ is a maximal stable set in $G_1$.
If $u,v\in V(G_1)$, then the triangle condition for $G_1$
implies the existence of a vertex $s\in S$ such that $\{u,v,s\}$ induces a triangle in $G_1$.
So we may assume that $v\in V(G_2)$. Notice that $u$ has a neighbor $s$ in $S$: if $u\in V(G_1)$ then this is true by the maximality of $S$; if $u\in V(G_2)$
then this follows from the definition of the join. Moreover, since $s\in V(G_1)$ and $v\in V(G_2)$, $s$ is adjacent to $v$.
In either case, $\{u,v,s\}$ induces a triangle in $G'$, hence $G'$ is triangle.
\qed
\end{proof}

We now introduce a further generalization of the triangle property.
A family of maximal stable sets ${\cal S}$ in a graph $G$ is said to be {\em non-edge covering} if every two distinct non-adjacent vertices $x,y\in V(G)$ are contained in a stable set from ${\cal S}$.

\begin{definition}\label{def:wt}
A graph $G$ is said to be {\em weakly triangle} if there exists a
non-edge covering collection ${\cal S}$ of maximal stable sets of $G$ that satisfies
the following {\em triangle property}:
for every $S\in{\cal S}$ and every pair $u,v$ of adjacent vertices in $V(G)\setminus S$, vertices $u$ and $v$ have a common neighbor in $S$.
\end{definition}

It follows directly from the definitions that every triangle graph is weakly triangle\label{triangle-weakly-triangle} (just take ${\cal S}$ to be the family of all maximal stable sets of $G$).
If every non-edge in a graph $G$ is contained in a unique maximal stable set, then $G$ is weakly triangle if and only if $G$ is triangle.
In particular, this implies that the $4$-vertex path $P_4$ is not weakly triangle.\label{ex:P4}

A clique $C$ in a graph $G$ is {\em simplicial} if there exists a vertex $v\in V(G)$ such that
$C = N[v]$. A graph $G$ is said to be {\em edge simplicial} if every edge of it is contained in a simplicial clique~\cite{CJ2006}.
As shown by Cheston et al.~in~\cite{ChesHaLas}, the edge simplicial graphs are equivalent to
the so-called upper bound graphs introduced by McMorris and Zaslavsky in 1982~\cite{MZ-1982}.
A graph $G$ is a {\em upper bound} graph if there exists a finite
partially ordered set $(V(G),\le)$ such that two distinct vertices $u,v$ of $G$ are adjacent
if and only if there exists $w\in V(G)$ such that $u\le w$ and $v\le w$.

For later use, we mention this equivalence explicitly.

\begin{proposition}[Cheston et al.~\cite{ChesHaLas}]\label{prop:edge-simplicial-upper-bound}
A graph $G$ is edge simplicial if and only if it is upper bound.
\end{proposition}

As will be argued in Section~\ref{sec:CIS}, edge simplicial graphs are related to general partition graphs as follows.

\begin{proposition}\label{prop:es-gpg}
Every edge simplicial graph is a general partition graph.
\end{proposition}

{The inclusion relations between the six considered graph classes can be summarized as follows:
\begin{eqnarray*}
&&\emph{edge simplicial graphs}\\
&\subset &\emph{general partition graphs} \\
&\subset &\emph{strongly equistable graphs} \\
&\subset &\emph{equistable graphs}\\
&\subset &\emph{triangle graphs}\\
&\subset &\emph{weakly triangle graphs}\,.
\end{eqnarray*}
All the inclusions are proper (see Section~\ref{sec:relations} for more details).}

\section{Two generalizations of CIS graphs}\label{sec:CIS}

{\bf Self-complementary graph properties.}
Self-complementary graph properties will play a particular role in this and subsequent sections.
A {\em graph property} (or: a {\em graph class}) is a set of graphs closed under isomorphism.
We will not distinguish between ``$G$ has property ${\cal P}$'' and ``$G\in {\cal P}$''.
For a graph property ${\cal P}$, we write that a graph is co-${\cal P}$ if its complement $\overline{G}$
has property ${\cal P}$, $\cap$-${\cal P}$ if
$G$ is ${\cal P}$ and co-${\cal P}$, and $\cup$-${\cal P}$ if $G$
is ${\cal P}$ or co-${\cal P}$.
A property ${\cal P}$ is said to be {\em self-complementary} if ${\cal P} = $co-${\cal P}$.
If ${\cal P}$ is a self-complementary graph property then clearly
$${\cal P} = \textrm{co-}{\cal P}= \cap\textrm{-}{\cal P} = \cup\textrm{-}{\cal P}\,.$$
Several well-known graph families are self-complementary, including
perfect graphs,
split graphs, 
threshold graphs, 
cographs, 
CIS graphs, 
and almost CIS graphs.

Recall the definition of CIS graphs.

\begin{definition}
\label{dfn:CIS}
We say that a graph $G$  is {\em CIS} if
every maximal clique $C$ and every maximal stable set
$S$ in $G$ intersect, that is, $C \cap S \neq \emptyset$.
\end{definition}

A {\it strong clique} in a graph $G$ is a clique that meets all maximal stable sets of $G$.
Similarly, a {\it strong stable set} is a stable set that meets all maximal cliques.
Notice that a graph is CIS if and only if every maximal clique is strong, or, equivalently,
if every maximal stable set is strong.

We now introduce two generalizations of CIS graphs.
Recall that a family of maximal stable sets ${\cal S}$ in a graph $G$ is called {\em non-edge covering} if every two distinct non-adjacent vertices $x,y\in V(G)$ are contained in a stable set from ${\cal S}$.
Similarly, a family of maximal cliques ${\cal C}$ is said to be {\em edge covering} if every two adjacent vertices $x,y\in V(G)$ are contained in a clique from ${\cal C}$.
Two families of sets ${\cal A}$, ${\cal B}$ are said to be {\em cross-intersecting} if
$A\cap B\neq\emptyset$ for every $A\in {\cal A}$ and $B\in {\cal B}$.

\begin{definition}
A graph $G$ is said to be {\em weakly CIS} if there exists
a cross-intersecting pair ${\cal C}$, ${\cal S}$ of families of subsets of $V(G)$ such that
${\cal C}$ is an edge covering collection of maximal cliques
and ${\cal S}$ is a non-edge covering collection of maximal stable sets.
\end{definition}

\begin{definition}
A graph $G$ is {\em semi-weakly CIS} if it admits an edge covering family of strong cliques.
\end{definition}

\begin{remark}
It follows directly from the definition that the weakly CIS property is self-complementary. On the other hand,
the semi-weakly CIS property is not self-complementary. To see this, observe that the $3$-sun, that is,
the graph obtained from the complete graph $K_3$ by pasting a new triangle along each edge of the $K_3$ is
semi-weakly CIS (in fact, edge simplicial), but not co-semi-weakly CIS.
\end{remark}

Recall that semi-weakly CIS graphs appeared in the literature under the name of {\it $(2,\ell)$-CIS graphs}
(for large enough $\ell$)~\cite{Gur11}. The following is a reformulation of a result from~\cite{MRDT}.

\begin{proposition}[McAvaney-Robertson-DeTemple~\cite{MRDT}]\label{prop:semi-weakly-CIS=gpg}
A graph $G$ is semi-weakly CIS if and only if $G$ is a general partition graph.
\end{proposition}

It can be seen that every simplicial clique is strong. Therefore:

\begin{proposition}\label{prop:edge-simplicial-semi-weakly-CIS}
Every edge simplicial graph is semi-weakly CIS.
\end{proposition}

In particular, together with Proposition~\ref{prop:semi-weakly-CIS=gpg}, this implies
Proposition~\ref{prop:es-gpg} from Section~\ref{sec:further}.

We now state a closure property of semi-weakly CIS graphs that we will use in later sections.

\begin{proposition}\label{prop:semi-weakly-CIS-disjoint-union-join}
If $G_1$ and $G_2$ are semi-weakly CIS, then their disjoint union $G_1+G_2$ and join $G_1*G_2$ are also semi-weakly CIS.
\end{proposition}

\begin{proof}
We proceed similarly as in the proof of Proposition~\ref{prop:triangle-disjoint-union-join}, making use of relations \eqref{eq1} and~\eqref{eq2}.
Suppose that $G_1$ and $G_2$ are semi-weakly CIS, and let ${\cal C}_1$ (resp.,~${\cal C}_2$) be
a collection of strong cliques covering all edges and all vertices of $G_1$ (resp.,~$G_2$).
It follows from~\eqref{eq1} that ${\cal C}_1\cup {\cal C}_2$
is an edge-covering collection of strong cliques of $G_1+G_2$, hence
$G_1+G_2$ is semi-weakly CIS.
Similarly,~\eqref{eq2} implies that
${\cal C} :=\{C_1\cup C_2\mid C_1\in {\cal C}_1\,, C_2\in {\cal C}_2\}$
is a collection of strong cliques covering all edges of $G_1*G_2$. (Notice that the assumption that
${\cal C}_1$ (resp.,~${\cal C}_2$) covers all {\it vertices} of $G_1$ (resp.,~$G_2$) is needed so that
the cliques in ${\cal C}$ cover all $V(G_1)$-$V(G_2)$ edges.) Therefore, $G_1*G_2$ is semi-weakly CIS.\qed
\end{proof}

The classes of CIS, semi-weakly CIS and weakly CIS graphs are nested as follows.

\begin{proposition}\label{prop:CIS-sw-CIS-w-CIS}
Every CIS graph is semi-weakly CIS. Every semi-weakly CIS is weakly CIS.
\end{proposition}

\begin{proof}
If $G$ is CIS then the set of all maximal cliques is an edge covering family of strong cliques.

If $G$ is semi-weakly CIS, with an edge covering family ${\cal C}$ of strong cliques, then
${\cal C}$ and ${\cal S}(G)$, the set of all maximal stable sets of $G$, form
a cross-intersecting pair of families of subsets of $V(G)$ such that
${\cal C}$ is an edge covering collection of maximal cliques
and ${\cal S}(G)$ is a non-edge covering collection of maximal stable sets,
showing that $G$ is weakly CIS.\qed
\end{proof}

Notice that since the CIS property is self-complementary, every CIS graph is also
co-semi-weakly CIS, that is, it admits a non-edge covering family of strong maximal stable sets.
Weakly CIS and semi-weakly CIS graphs were also considered under different names in~\cite{Gur11}.
Replacing in the definition of weakly CIS graphs
each of the conditions ``edge covering'' and ``non-edge covering'' with the weaker condition
``vertex covering'' results in the class of so-called {\em normal graphs}, a generalization of perfect graphs
introduced by K\"orner in~\cite{Koerner73} and
studied in a series of papers~\cite{BW,CKLMS,DeSimoneKoerner99,FachKoerner07,KM,Patakfalvi,Wagler06,Wagler08}.
Clearly, the classes of weakly CIS and normal graphs are related as follows.

\begin{proposition}\label{prop:weakly-CIS-normal}
Every weakly CIS graph is normal.
\end{proposition}

For later use, we also explicitly state the relation between perfect and normal graphs.

\begin{proposition}[K\"orner~\cite{Koerner73}]\label{prop:perfect-normal}
Every perfect graph is normal.
\end{proposition}

Proposition~\ref{prop:semi-weakly-CIS=gpg} together with
Theorems~\ref{thm:gpg-seq}, \ref{thm:Seq-eq} and \ref{thm:triangle} implies the following result
(which is also easy to prove directly).

\begin{proposition}\label{prop:semi-weakly-CIS-triangle}
Every semi-weakly CIS graph is triangle.
\end{proposition}

We now show the following relation between weakly CIS and weakly triangle graphs.

\begin{proposition}\label{prop:weakly-CIS-cap-weakly-triangle}
Every weakly CIS graph is $\cap$-weakly triangle.
\end{proposition}

\begin{proof}
Let $G$ be a weakly CIS graph, and let ${\cal C}$, ${\cal S}$ be a cross-intersecting pair
of families of subsets of $V(G)$ such that
${\cal C}$ is an edge covering collection ${\cal C}$ of maximal cliques
and ${\cal S}$ is a non-edge covering collection of maximal stable sets.

We claim that ${\cal S}$ has the triangle property.
Indeed, let $S\in{\cal S}$ and let $u,v$ be a pair of adjacent vertices in $V(G)\setminus S$.
Then, there is a $C\in {\cal C}$ containing $u$ and $v$. Since ${\cal C}$ and ${\cal S}$
are cross-intersecting, $C\cap S = \{w\}$ for some $w\in V(G)$. Clearly, $w$ is a common
neighbor of $u$ and $v$ in $S$.
By symmetry, ${\cal C}$ has the co-triangle property, which completes the proof.\qed
\end{proof}

\section{A characterization of split graphs}\label{sec:split}

There is another class of graphs closely related to CIS graphs: the so-called almost CIS graphs.

\begin{definition}\label{dfn:almost-CIS}
We say that a graph $G$  is {\em almost CIS} if
every maximal clique $C$ and every maximal stable set
$S$ in $G$ intersect, except for a unique pair,
$C_0$  and  $S_0$.
\end{definition}

The class of almost CIS graphs was introduced by Boros et al.~in~\cite{BGZ09} and
characterized by Wu~et al.~\cite{WZZ}. The characterization is based on split graphs.

Now we can state the characterization of almost CIS graphs
conjectured by~\cite{BGZ09} and proved in \cite{WZZ}.

\begin{proposition}[Wu et al.~\cite{WZZ}]\label{prop:Wu}
A graph  $G = (V,E)$ is almost CIS if and only if $V = C_0 \cup S_0$, where
$C_0$ is a maximal clique, $S_0$ is a maximal stable set, and  $C_0 \cap S_0 = \emptyset$;
or, in other words, if and only if $G$ is a split graph with
a unique split partition.
\end{proposition}

Proposition~\ref{prop:perfect-normal} and the fact that split graphs are perfect implies
the following.

\begin{corollary}\label{cor:almost-CIS-normal}
Every almost CIS graph is normal.
\end{corollary}

We find it useful to introduce a name for the graph property, which is the union of CIS and almost CIS graphs.

\begin{definition}\label{dfn:quasi-CIS}
We say that a graph $G$  is {\em quasi CIS} if it is CIS or almost CIS.
\end{definition}

We also recall the following property of split graphs.

\begin{proposition}[Boros et al.~\cite{BGZ09}]\label{prop:split-quasi-CIS}
Every split graph is quasi CIS.
\end{proposition}

We will now give another characterization of split graphs. We start with the following proposition.

\begin{proposition}
Every CIS split graph is $\cap$-edge simplicial.
\end{proposition}

\begin{proof}
Let $G$ be a CIS split graph. Fix a split partition $(C,S)$ of $G$ such that $S$ is a maximal stable set.
Since $G$ is CIS, $C$ is not a maximal clique. Therefore, there exists a vertex $v\in S$ that is adjacent to all vertices in $C$. But then every edge $xy\in E(G)$ with $x,y\in C$ is contained in the simplicial clique $N[v]$.
Since every edge with an endpoint, say $x$, in $S$ is contained in the simplicial clique $N[x]$, this implies that $G$ is edge simplicial.

Since both split and CIS are self-complementary properties, applying the above argument to the complement of $G$
shows that $G$ is co-edge simplicial, and thus $\cap$-edge simplicial.\qed
\end{proof}

\begin{proposition}\label{prop:cap-es-split}
Every $\cap$-edge simplicial graph is split.
\end{proposition}

\begin{proof}
Let $G$ be a $\cap$-edge simplicial graph. Let
$\{C_1,\ldots, C_k\}$ be the set of all simplicial cliques of $G$, and
$\{S_1,\ldots, S_\ell\}$ be the set of all co-simplicial stable sets of $G$
(where a stable set is co-simplicial in $G$ if it is a simplicial clique in $\overline{G}$).
Choose a simplicial vertex $u_i\in C_i$ and a co-simplicial vertex
$v_j\in S_j$ for all $i\in [k]:= \{1,\ldots,k\}$ and all $j\in [\ell]$.
Let $S = \{u_1,\ldots, u_k\}$ and $C = \{v_1,\ldots, v_\ell\}$.

We claim that $S$ is a stable set. Indeed, no two vertices in $S$ are adjacent,
since otherwise they would define the same simplicial clique.
Similarly, $C$ is a clique.

Furthermore, every vertex $w\in V(G)\setminus (C\cup S)$ satisfies either $C\subseteq N(w)$ or $S\cap N(w) = \emptyset$.
Indeed, if there exists a vertex $w\in V(G)\setminus (C\cup S)$ that is adjacent to some
$u_j\in S$ and non-adjacent to some $v_j\in C$, then  $v_j\neq u_i$; moreover, if $v_j$ is adjacent to $u_i$ then $u_i$ cannot be simplicial, while if
$v_j$ is non-adjacent to $u_i$ then $v_j$ cannot be co-simplicial. In each case, we get a contradiction.

Thus, we can define two sets $K$ and $L$ where
$$K = \{w\in V(G)\setminus (C\cup S)\mid C\subseteq N(w)\}$$
and
$$L = \{w\in V(G)\setminus (C\cup S)\mid N(w)\cap S = \emptyset\}$$
such that $V(G) = (C\cup K)\cup (S\cup L)$.

Let us now show that $K$ is a clique and $L$ is a stable set. This will imply that
$(C\cup K, (S\cup L)\setminus (C\cup K))$ is a split partition of $G$.
Suppose that two distinct vertices $w, w'\in K$ are non-adjacent.
Since $G$ is co-edge simplicial, there exists a co-simplicial vertex $v_j\in C$ such that
the corresponding co-simplicial stable set $S_j$ contains $w$ and $w'$. Since $v_j\in S_j$, vertex $w$ is non-adjacent to
$v_j$, contrary to the fact that $v_j\in C\subseteq N(w)$.
Similarly, we can show that $L$ is a stable set, which completes the proof.
\qed
\end{proof}

From the above three propositions, we obtain the following characterization of split graphs.

\begin{theorem}\label{thm:split-characterization}
A graph $G$ is split if and only if $G$ is almost CIS or $\cap$-edge simplicial.
\end{theorem}

\section{Random and explicit constrictions of split $\cap$-edge simplicial non-CIS graphs}\label{sec:split-constructions}

We now show that in a natural random model split graphs are asymptotically almost surely $\cap$-edge simplicial but not CIS.

Consider the following construction of a random split graph $G_{k,\ell}$ on $k+\ell$ vertices, where $k$ and $\ell$ are positive integers. Take two disjoint sets $C$ and $S$ with $k$ and $\ell$ vertices, respectively.
Make $C$ a clique and $S$ a stable set. For every pair $(c,s)$ where $c\in C$ and $s\in S$, make $c$ and $s$ adjacent with probability $1/2$. Make all these $k\ell$ choices independently of each other.

\begin{proposition}\label{lem:random-split}
For every $\epsilon>0$ there exist integers $k_0$ and $\ell_0$ such that
for all $k\ge k_0$ and for all $\ell\ge \ell_0$, graph $G_{k,\ell}$ has all of
the following properties with probability at least $1-\epsilon$:
\begin{enumerate}[(1)]
  \item $S$ is a maximal stable set.
  \item $C$ is a maximal clique.
  \item Every two vertices of $C$ have a common neighbor in $S$.
  \item Every two vertices of $S$ have a common non-neighbor in $C$.
\end{enumerate}
\end{proposition}

\begin{proof}
Let us denote by $A_S$, $A_C$, $A_{CS}$ and $A_{SC}$ the
four random events associated with properties $(1)$--$(4)$ above (in this order).
We need to show that the probability of the event
$A_SA_CA_{CS}A_{SC}$ is arbitrarily close to $1$ for all
large enough values of $k$ and $\ell$,
or, equivalently, that $\Pr(\overline{A_S}\cup \overline{A_C}\cup \overline{A_{CS}}\cup \overline{A_{SC}})$ is arbitrarily small for all large enough values of $k$ and $\ell$,
where $\overline{X}$ denotes the complementary event of
an event $X$.

Fix some positive integers $k,\ell$. First, notice that $S$ is {\it not} a maximal stable set in $G_{k,\ell}$ if and only if there exists a vertex in $C$ that is non-adjacent to all vertices in $S$. The probability of this event is at most the sum, over all $c\in C$, of the probability that vertex $c$ is non-adjacent to all vertices in $S$, which is $\left(\frac{1}{2}\right)^\ell$. Hence, $\Pr(\overline{A_S})\le k\left(\frac{1}{2}\right)^\ell$.
By a similar argument, we obtain $\Pr(\overline{A_C})\le \ell\left(\frac{1}{2}\right)^k$.

Now, let us estimate the probability $\Pr(\overline{A_{CS}})$ of the event
that there exists a pair of vertices $c,c'\in C$ that have no common neighbor in $S$.
This probability is not more than the sum, over all unordered pairs of distinct vertices $c,c'\in C$,
of the probability $p_{cc'}$ that $c$ and $c'$ have no common neighbor in $S$.
The event that $c$ and $c'$ have no common neighbor in $S$
is the intersection of $\ell$ independent events that $c$ and $c'$ are non-adjacent to some vertex $s$, over all $s\in S$.
For a fixed $s\in S$, the probability that $c$ and $c'$ are non-adjacent to $s$,
is $1-\left(\frac{1}{2}\right)^2 = \frac{3}{4}$. Hence,
$$p_{cc'} = \left(\frac{3}{4}\right)^\ell~\textrm{~~~~and consequently~~~~}\Pr(\overline{A_{CS}})\le {k\choose 2}\left(\frac{3}{4}\right)^\ell\,.$$
By a similar argument, we obtain
$\Pr(\overline{A_{SC}})\le {\ell\choose 2}\left(\frac{3}{4}\right)^k\,.$
Therefore,
\begin{eqnarray*}
\Pr(\overline{A_S}\cup \overline{A_C}\cup \overline{A_{CS}}\cup \overline{A_{SC}})&\le & \Pr(\overline{A_S})+\Pr(\overline{A_C})+\Pr(\overline{A_{CS}})+\Pr(\overline{A_{SC}})\\
& \le & k\left(\frac{1}{2}\right)^\ell+\ell\left(\frac{1}{2}\right)^k+{k\choose 2}\left(\frac{3}{4}\right)^\ell+{\ell\choose 2}\left(\frac{3}{4}\right)^k\,.\\
\end{eqnarray*}
For every positive $\epsilon$, this expression is smaller than $\epsilon$ for all
large enough values of $k$ and $\ell$.  This completes the proof. \qed
\end{proof}

Notice that in a split graph $G$ with a split partition $(C,S)$, every vertex in $S$ is simplicial. In particular, if every two vertices of $C$ have a common neighbor in $S$, then every edge in $G$ is contained in a simplicial clique.
Moreover, if $S$ is a maximal stable set and $C$ is a maximal clique then $G$ is not CIS.
Hence, Lemma~\ref{lem:random-split} has the following corollary.

\begin{corollary}\label{cor:random}
Asymptotically almost surely, $G_{k,\ell}$ is $\cap$-edge-simplicial but not CIS.
\end{corollary}

The above proof is not constructive. An explicit infinite family of $\cap$-edge-simplicial non-CIS split graphs
can be obtained as follows.
Let $\Pi$ be a finite projective plane of order $n\ge 2$, defined by a set of $q$ points $P$ and a
set of $q$ lines $L$, where $q=n^2+n+1$; see, e.g.,~\cite{HP73}.
Let us associate to $\Pi$ the split graph $G(\Pi)$
whose vertices are the points and the lines of $\Pi$, the pair $(P,L)$ is a split partition
of $G(\Pi)$ (in particular, $P$ is a clique and $L$ is a stable set), and a vertex $p\in P$
is adjacent to a vertex $\ell\in L$ if and only if the corresponding point and line are incident in $\Pi$.

\begin{proposition}\label{prop:projective}
For every finite projective plane $\Pi$, the split graph $G(\Pi)$ is $\cap$-edge simplicial but not CIS.
\end{proposition}

\begin{proof}
Obviously, $P$ is a maximal clique in $G(\Pi)$ and $L$ is a maximal stable set. Indeed, this follows from the fact that
no line of $\Pi$ contains all points of $\Pi$, and no point belongs to all lines.
Therefore, as $P\cap L= \emptyset$, graph $G(\Pi)$ is not CIS.
Furthermore, $G(\Pi)$ is edge simplicial, since every two points of $\Pi$ belong to a (unique) line, and
co-edge simplicial, since no two lines cover the whole point set of $\Pi$.\qed
\end{proof}

\section{Inclusion relations between considered graph classes}\label{sec:relations}

In this section, we summarize the known inclusion relations between the properties considered in this paper.
Figure~\ref{fig:Hasse} shows the Hasse diagram of the known inclusion relations between these properties,
including their relations to some other well-known self-complementary graph classes (threshold graphs,
cographs, perfect graphs, normal graphs). Recall that notation $\cap$-${\cal C}$
denotes the class $\{G\mid G\in {\cal C} \wedge \overline{G}\in {\cal C}\}$, while
$\cup$-${\cal C}$ denotes the class $\{G\mid G\in {\cal C} \vee \overline{G}\in {\cal C}\}$.

\begin{figure}[h!]
\begin{center}
\includegraphics[width=\textwidth]{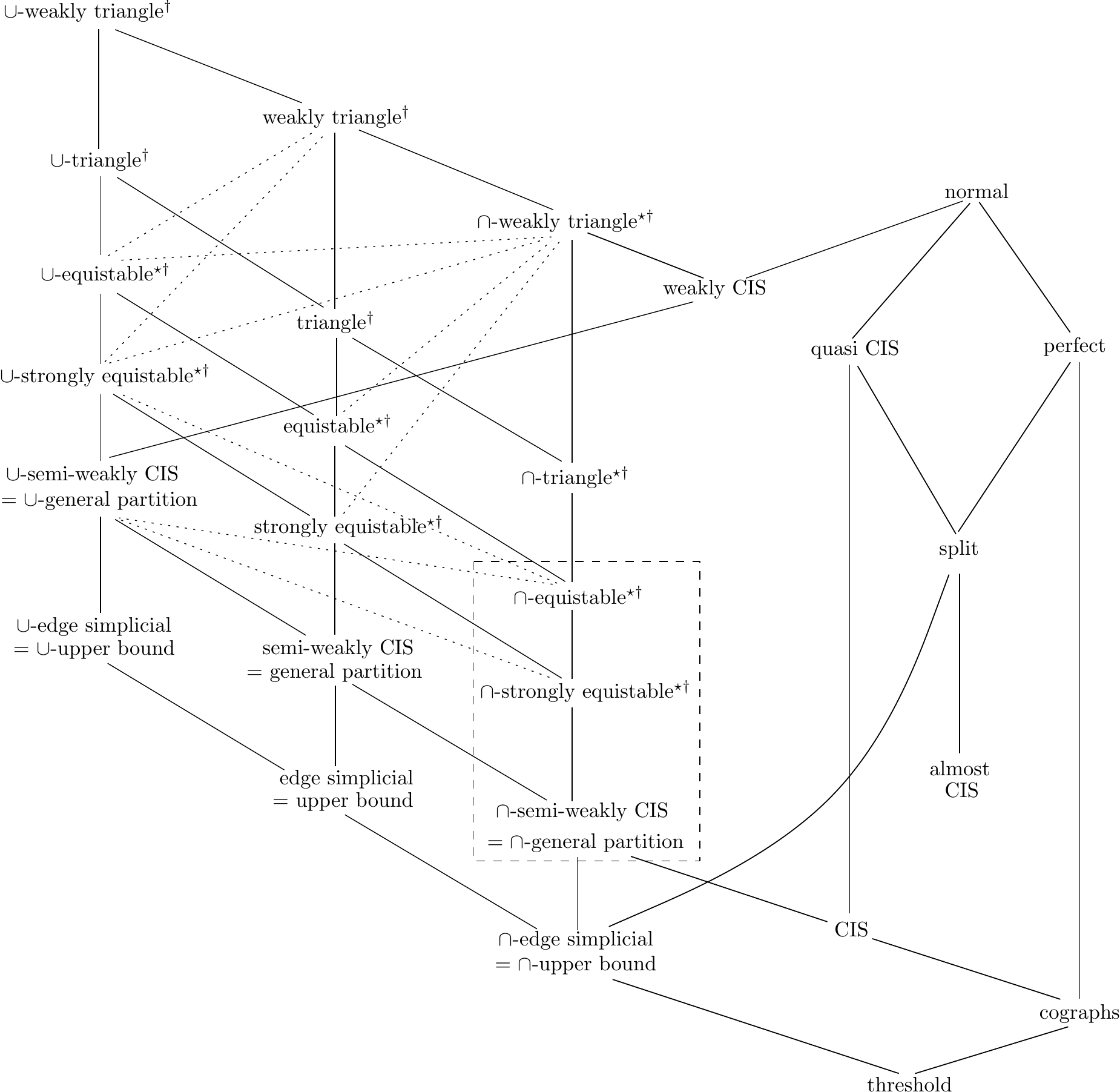}
\caption{Hasse diagram of the known inclusion and non-inclusion relations between considered properties.
A class $X$ is contained in class $Y$ if and only if there is an upward-going $X$-$Y$ path in the diagram, except possibly for the following:
(1) It is not known whether any of the inclusion relations between the three classes in the rectangular region is proper or not;
(2) For each graph class $X$ marked with $^\star$, it is not known whether $X$ is contained in
the class of weakly CIS graphs or not;
 (3) For each graph class marked with $^\dagger$, it is not known whether $X$ is contained in
the class of normal graphs or not;
(4) For each dotted line, it is not known whether it represents an inclusion relation or not.}
\label{fig:Hasse}
\end{center}
\end{figure}

\begin{sloppypar}
In what follows, we justify the inclusion and non-inclusion relations depicted by the Hasse diagram in Figure~\ref{fig:Hasse},
traversing the diagram bottom up. {For each class $X$, we justify only the inclusion relations of the form $X\subseteq Y$ where $Y$ is an immediate successor of $X$ in the diagram, and the non-inclusion relations of the form $X\nsubseteq Y$ where $Y$ is a maximal class not containing $X$. This is sufficient, since the remaining relations follow by transitivity.} For each inclusion relation $X\subseteq Y$, we also state whether it is proper (whenever known),
and justify cases $X\nsubseteq Y$ (whenever known).
For two graph classes $X$ and $Y$ such that $X\nsubseteq Y$, a graph belonging to $X\setminus Y$ is said to be a {\it $(X,Y)$-separating graph}.
Several pairs $(X,Y)$ in our Hasse diagram such that $X\nsubseteq Y$ have common separating graphs.
For convenience, we list here all separating graphs that will occur in our analysis:
\begin{enumerate}[--]
  \item $2K_2$, the disjoint union of two copies of $K_2$,
  \item 
  the {\it bull}, that is, the graph obtained from $P_4$ by gluing a triangle on the middle edge (formally, the graph with vertex set $\{a,b,c,d,e\}$ and edge set $\{ab,bc,cd,be,ce\}$),
  \item $C_4$, the $4$-cycle (notice that $C_4$ is the complement of $2K_2$),
  \item $C_9$, the $9$-cycle,
  \item ${\it CK}$, the disjoint union of $C_4$ and $K_2$,
  \item $C_5^*$, the $10$-vertex graph obtained from the $5$-cycle $C_5$ by gluing a new triangle along every edge,
  \item $\textrm{Cir}_9$, a $9$-vertex circle graph\footnote{A circle graph is the intersection graph of chords on the circle.}
  that can be defined as follows:
$V(\textrm{Cir}_9)=\{1,\ldots,9\}$ and the set of its maximal stable set is given by ${\cal S}(\textrm{Cir}_9)
 = \{\{1,2,3\},\{4,5,6\},\{7,8,9\},\{1,4,7\},\{3,6,9\}\}$.
  \item $F$, the split incidence graph of the Fano plane $\Pi_F$ (the finite projective plane of order $2$), that is, $F$ is the $14$-vertex split graph with split partition
$(P,L)$ (in particular, $P$ is a clique and $L$ is an independent set), where
$P$ is the ($7$-element) set of points of the Fano plane,
$L$ is the ($7$-element) set of lines of the Fano plane,
and a vertex $p\in P$
is adjacent to a vertex $\ell\in L$ if and only if the corresponding point and line are incident in $\Pi_F$,
  \item ${\it FK}$, the disjoint union of $F$ and $K_2$,
  \item $G_{12}$, a specific $12$-vertex graph (that will be defined in Section~\ref{sec:G12}),
  \item $K_1$, the $1$-vertex graph,
  \item $L\overline{L}$, the disjoint union of graph $L$ and its complement, where $L$ is
  the graph obtained from the line graph of $K_{5,6}$ by gluing a new triangle along every edge,
  \item $L(K_{3,3})$, the line graph of the complete bipartite graph $K_{3,3}$,
  \item the {\it net}, that is, the complement of $S_3$,
  \item $P_4$, the $4$-vertex path,
  \item $S_3$, the $3$-sun, that is, the split graph with vertex set $\{v_1,v_2,v_3,v_{12},v_{13},v_{23}\}$ in which
  $\{v_1,v_2,v_3\}$ is a clique, $\{v_{12},v_{13},v_{23}\}$ is an independent set, and $N(v_{ij}) = \{v_i,v_j\}$ for all $(i,j)\in \{(1,2),(1,3),(2,3)\}$,
  \item ${\it SK}$, the disjoint union of $S_3$ and $K_2$,
  \item {$G_{14}$, the $14$-vertex graph that is the complement of the line graph of a particular $9$-vertex triangle-free graph $G^*$, see~Fig.~\ref{fig:example},}

\begin{figure}[h!]
\begin{center}
\includegraphics[width=0.6\textwidth]{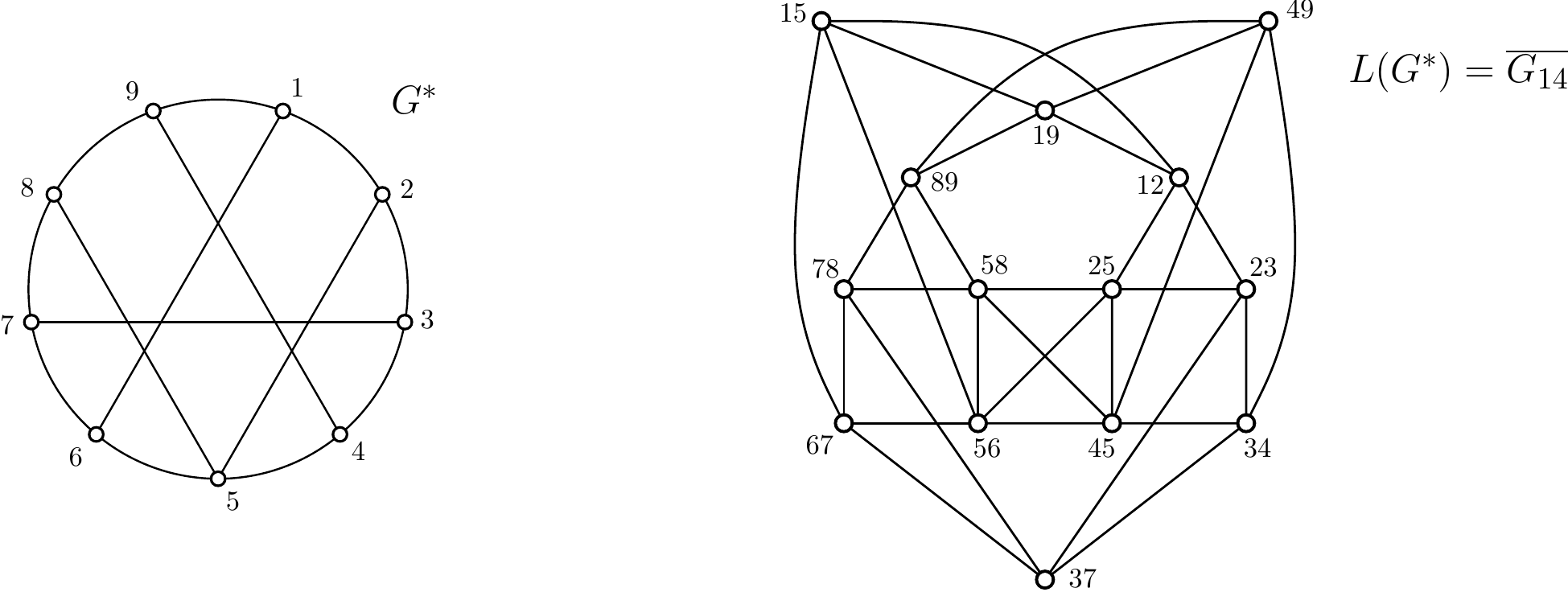}
\caption{The graph $G^*$ and its line graph. $G_{14}$ is the complement of $L(G^*)$.}
\label{fig:example}
\end{center}
\end{figure}

  \item {$G_{22}$, the $22$-vertex graph that is the complement of the line graph of the circulant $C_{11}(\{1,3\})$, see~\cite{MTT} for more details.}
\end{enumerate}
\end{sloppypar}

\subsection{Threshold graphs}
We recall the following characterization of threshold graphs due to Chv\'atal and Hammer~\cite{CH77}.
\begin{sloppypar}
\begin{theorem}\label{thm:ChvatalHammer}
For every graph $G$ the following conditions are equivalent:
\begin{enumerate}[(a)]
  \item\label{itemCHa} $G$ is threshold.
  \item\label{itemCHb} $G$ has no induced subgraph isomorphic to $2K_{2},C_{4}$ or $P_{4}$.
  \item\label{itemCHc} There is a partition of $V$ into two disjoint sets $I,K$ such that $I$ is an independent set, $K$ is a clique, and there exists an ordering $u_{1},u_{2},\ldots,u_k$ of vertices in $I$ such that $N(u_{1})\subseteq N(u_{2})\subseteq\ldots\subseteq N(u_k)$.
\end{enumerate}
\end{theorem}
\end{sloppypar}

In part (c) above, we may assume without loss of generality that $I$ is a maximal independent set of $G$, that is, that $\cup_{u\in I}N(u) = K$.

\begin{corollary}
Every threshold graph is edge simplicial and co-edge simplicial. The inclusion is proper.
\end{corollary}

\begin{proof}
Let $G$ be a threshold graph.
The equivalence between \eqref{itemCHa} and \eqref{itemCHc} in Theorem~\ref{thm:ChvatalHammer}
implies that the simplicial cliques of the form $N[u_1],\ldots, N[u_k]$ cover all edges of $G$.
Therefore, $G$ is edge simplicial.
The equivalence between \eqref{itemCHa} and \eqref{itemCHb} in Theorem~\ref{thm:ChvatalHammer} shows that the class of threshold graphs is self-complementary.
Therefore, if $G$ is threshold, then $G$ is also co-edge simplicial.
To see that the inclusion is proper, consider for example the bull. It can be seen that this is a (self-complementary)
edge simplicial co-edge simplicial graph that is not threshold since it contains an induced $P_4$.
\qed
\end{proof}
\newline\noindent The fact that every threshold graph is a cograph is well known, see, e.g.,~\cite{BLS}; this inclusion is also proper
(e.g., $C_4$ is a cograph that is not threshold).

\subsection{Cographs}

The fact that every cograph is perfect is well known, see, e.g.,~\cite{BLS}; the inclusion is proper
(e.g., $P_4$ is a perfect graph that is not a cograph).
The fact that every cograph (equivalently: every $P_4$-free graph) is CIS is also well known, see, e.g.,~\cite{Gur11}; the inclusion is proper
(e.g., the bull is a CIS graph that is not a cograph).\\
Regarding the remaining non-inclusions, it is sufficient to justify the following two:
\begin{itemize}
  \item Not every cograph is $\cup$-edge simplicial, for example, ${\it CK}$, that is, the graph $C_4+K_2$, is not.
  \item Not every cograph is split. For example, $C_4$ is not.
\end{itemize}

\subsection{$\cap$-edge simplicial ($\cap$-upper bound) graphs}

The fact that the $\cap$-edge simplicial graphs are exactly the $\cap$-upper bound graphs
follows from Proposition~\ref{prop:edge-simplicial-upper-bound}.
The fact that every $\cap$-edge simplicial graph is edge simplicial follows from definition; the inclusion is proper
(e.g., $2K_2$ is edge simplicial but not $\cap$-edge simplicial).
The fact that every $\cap$-edge simplicial graph is $\cap$-semi-weakly-CIS follows from Proposition~\ref{prop:edge-simplicial-semi-weakly-CIS};
the inclusion is proper (e.g., $2K_2$ is $\cap$-semi-weakly-CIS but not $\cap$-edge simplicial).
The fact that every $\cap$-edge simplicial graph is split follows from Proposition~\ref{prop:cap-es-split};
the inclusion is proper (e.g., $P_4$ is split but not $\cap$-edge simplicial).
\\
Regarding the remaining non-inclusions, we have:
\begin{itemize}
  \item Not every $\cap$-edge simplicial graph is CIS. For example, $F$ is not.

  \item Not every $\cap$-edge simplicial graph is almost CIS. For example, $K_1$ is not.
\end{itemize}

\subsection{CIS graphs}

Since being CIS is a self-complementary property, the fact that every CIS graph is $\cap$-semi-weakly-CIS follows
from Proposition~\ref{prop:CIS-sw-CIS-w-CIS}.
This inclusion is proper: Proposition~\ref{prop:projective} implies that
the split incidence graph of any projective plane is $\cap$-edge-simplicial (and thus also $\cap$-semi-weakly-CIS, by Proposition~\ref{prop:edge-simplicial-semi-weakly-CIS})
but not CIS. The smallest example from this family is $F$, the
split incidence graph of the Fano plane. (Other examples are given by Corollary~\ref{cor:random}.)
The fact that every CIS graph is quasi CIS is part of Definition~\ref{dfn:quasi-CIS}, and the inclusion is proper (e.g., $P_4$ is quasi CIS but not CIS).
\\
Regarding the remaining non-inclusions, we have:
\begin{itemize}
  \item Not every CIS graph is $\cup$-edge simplicial (since not every cograph is $\cup$-edge simplicial).

  \item Not every CIS graph is perfect. For example, $C_5^*$ is CIS (every maximal clique is simplicial, hence strong)
  but not perfect (since it contains an induced $C_5$).
\end{itemize}

\subsection{Almost CIS graphs}

The fact that every almost CIS graph is split follows from Proposition~\ref{prop:Wu}; the inclusion is proper
(e.g., $K_1$ is split but not almost CIS).
\\
Regarding the remaining non-inclusions, observe that $P_4$ is a self-complementary graph that is almost CIS but not $\cup$-weakly triangle
(cf.~the discussion on p.~\pageref{ex:P4} after Definition~\ref{def:wt}).

\subsection{Split graphs}

The fact that every split graph is quasi CIS follows from Proposition~\ref{prop:split-quasi-CIS};
the inclusion is proper (e.g., $C_4$ is quasi CIS but not split). The fact that every split
graph is perfect is well known, see, e.g.,~\cite{BLS}; this inclusion is also proper
(e.g., $C_4$ is a perfect graph that is not split).
\\Regarding the remaining non-inclusions, $P_4$ is a split graph that is not $\cup$-weakly triangle.

\subsection{Edge simplicial (upper bound) graphs}

The fact that the edge simplicial graphs are exactly the upper bound graphs
follows from Proposition~\ref{prop:edge-simplicial-upper-bound}.
The fact that every edge simplicial graph is $\cup$-edge simplicial follows from definition; the inclusion is proper
(e.g., $C_4$ is $\cup$-edge simplicial but not edge simplicial).
The fact that every edge simplicial graph is semi-weakly-CIS
follows from Proposition~\ref{prop:edge-simplicial-semi-weakly-CIS};
the inclusion is proper
(e.g., $C_4$ is semi-weakly-CIS but not edge simplicial).
\\
Regarding the remaining non-inclusions, we have:
\begin{itemize}
  \item Not every edge simplicial graph is $\cap$-triangle. For example, the graph $S_3$ is not.
  Indeed, on the one hand, $S_3$ is edge simplicial (and hence also triangle), on the other hand, it is not co-triangle.
  The complement of $S_3$ is isomorphic to the {\it net}, that is, the graph
  with vertex set $\{x_1,x_2,x_3,y_1,y_2,y_3\}$ in which
  $\{x_1,x_2,x_3\}$ is a clique, $\{y_1,y_2,y_3\}$ is an independent set, and $N(y_{i}) = \{x_i\}$ for all $i\in \{1,2,3\}$.
  The maximal stable set $S = \{y_1,y_2,y_3\}$ does not contain any vertex adjacent to both endpoints of the edge $x_1x_2$, thus the complement of $\it net$ is
  not triangle.

  \item Not every edge simplicial graph is quasi CIS. For example, the graph ${\it SK}$, the disjoint union of $S_3$ and $K_2$, is not. Indeed, denoting
  the vertices of ${\it SK}$ as $V({\it SK}) = \{v_1,v_2,v_3,v_{12},v_{13},v_{23}, u_1,u_2\}$ where $u_1, u_2$ form the $2$-vertex component, and
  the remaining edges are as in the definition of $S_3$, the maximal clique $C = \{v_1,v_2,v_2\}$ is disjoint from two distinct
  maximal stable sets, namely $\{v_{12},v_{13},v_{23}, u_1\}$, and $\{v_{12},v_{13},v_{23}, u_2\}$.

  \item Not every edge simplicial graph is perfect. For example, $C_5^*$ is not.
\end{itemize}

\subsection{$\cup$-edge simplicial ($\cup$-upper bound) graphs}

The fact that the $\cup$-edge simplicial graphs are exactly the $\cup$-upper bound graphs
follows from Proposition~\ref{prop:edge-simplicial-upper-bound}. The fact that every $\cup$-edge simplicial graph is $\cup$-semi-weakly-CIS
follows from Proposition~\ref{prop:edge-simplicial-semi-weakly-CIS}; the inclusion is proper
(e.g., $L(K_{3,3})$ is a self-complementary CIS graph, hence also $\cup$-semi-weakly-CIS, that is not edge simplicial).
\\
Regarding the remaining non-inclusions, we have:
\begin{itemize}
  \item Not every $\cup$-edge simplicial graph is triangle.
  For example, the net is not. (The complement of the net, $S_3$, is edge simplicial, but the net is not triangle, as justified above.)

  \item Not every $\cup$-edge simplicial graph is quasi CIS (since not every edge simplicial graph is quasi CIS).

  \item Not every $\cup$-edge simplicial graph is perfect (since not every edge simplicial graph is perfect).
\end{itemize}

\subsection{$\cap$-semi-weakly CIS ($\cap$-general partition) graphs}

The fact that the $\cap$-semi-weakly CIS graphs are exactly the $\cap$-general partition graphs
follows from Proposition~\ref{prop:semi-weakly-CIS=gpg}.
The fact that every $\cap$-semi-weakly CIS graph is semi-weakly CIS follows from definition; the inclusion is proper
(e.g., $S_3$ is semi-weakly CIS but not co-triangle, hence also not
co-semi-weakly CIS and not $\cap$-semi-weakly CIS).
The fact that every $\cap$-semi-weakly CIS graph is $\cap$-strongly equistable
follows from Theorem~\ref{thm:gpg-seq}. We do not know whether the inclusion is proper or not.
\\
Regarding the remaining non-inclusions, we have:
\begin{itemize}
  \item Not every $\cap$-semi-weakly CIS graph is $\cup$-edge simplicial (since not every cograph is $\cup$-edge simplicial).

  \item Not every $\cap$-semi-weakly CIS graph is quasi CIS. For example, ${\it FK}$, the disjoint union of $F$ and $K_2$,
  is not. (Graph $F$ is self-complementary, which implies that the complement of ${\it FK}$ is isomorphic to
  the join of $F$ and $2K_1$. By Proposition~\ref{prop:projective}, graph $F$ is edge simplicial, hence also semi-weakly CIS.
  Similarly, $K_2$ and its complement $2K_1$ are semi-weakly CIS,
  hence Proposition~\ref{prop:semi-weakly-CIS-disjoint-union-join} implies that both ${\it FK}$
  and its complement are semi-weakly CIS. Therefore, ${\it FK}$ is $\cap$-semi-weakly CIS. On the other hand,
  ${\it FK}$ is not quasi CIS, since the maximal clique $C = P$ is disjoint from two different maximal stable sets
  of the form $L\cup \{v\}$ where $v\in V(K_2)$.)

  \item Not every $\cap$-semi-weakly CIS graph is perfect (since not every CIS graph is perfect).
\end{itemize}

\subsection{Semi-weakly CIS (general partition) graphs}

The fact that the semi-weakly CIS graphs are exactly the general partition graphs is the statement of Proposition~\ref{prop:semi-weakly-CIS=gpg}.
The fact that every semi-weakly CIS graph is $\cup$-semi-weakly CIS follows from definition; the inclusion is proper
(e.g., the net is co-edge simplicial, and thus co-semi-weakly CIS, but not triangle, hence also not
semi-weakly CIS).
The fact that every semi-weakly CIS graph is strongly equistable
follows from Theorem~\ref{thm:gpg-seq}; {the inclusion is proper since the graph $G_{22}$,
that is, the complement of the line graph  of the circulant $C_{11}(\{1,3\})$,
is strongly equistable but not semi-weakly CIS (see~\cite{MTT}).}
\\
Regarding the remaining non-inclusions, we have:
\begin{itemize}
  \item Not every semi-weakly CIS graph is $\cup$-edge simplicial (since not every cograph is $\cup$-edge simplicial).

  \item Not every semi-weakly CIS graph is $\cap$-triangle. For example, $S_3$ is not (it is not co-triangle).

  \item Not every semi-weakly CIS graph is quasi CIS (since not every $\cap$-semi-weakly CIS graph is quasi CIS).

  \item Not every semi-weakly CIS graph is perfect (since not every CIS graph is perfect).
\end{itemize}

\subsection{$\cup$-semi-weakly CIS ($\cup$-general partition) graphs}\label{sec:G12}

The fact that the $\cup$-semi-weakly CIS graphs are exactly the $\cup$-general partition graphs follows from Proposition~\ref{prop:semi-weakly-CIS=gpg}.
The fact that every $\cup$-semi-weakly CIS graph is weakly CIS
follows from Proposition~\ref{prop:CIS-sw-CIS-w-CIS} and the fact that the weakly CIS property is self-complementary.
To show that the inclusion is proper, we will now define and analyze a graph denoted by $G_{12}$ that is weakly CIS graph but not $\cup$-semi-weakly CIS.
We introduce the graph $G_{12}$ by the following families ${\cal C}$ and ${\cal S}$ of cliques and stable sets:
\begin{eqnarray*}
  {\cal C} &=&
  \{\{1,4,7\},\{2,4,5,6\},\{2,4,6,7\},\{2,4,6,9\},\{2,4,9,12\},\{2,5,8\},\{2,6,7,11\},\\
   && ~~\{2,11,12\},\{3,6,9\},\{4,5,6,10\},\{4,10,12\},
  \{6,10,11\},
  \{10,11,12\}\}\,, \\
  {\cal S} &=& \{ \{1,2,3,10\},\{1,3,5,11\},\{1,3,5,12\},\{1,3,8,10\},\{1,3,8,11\},\{1,3,8,12\},\{1,5,9,11\},\{1,6,8,12\},\\
  &&~~\{1,8,9,10\},\{1,8,9,11\},
  \{3,4,8,11\},\{3,5,7,12\}, \{3,7,8,10\},\{3,7,8,12\},\{5,7,9\},\{7,8,9,10\}\}\,.
  \end{eqnarray*}
It is tedious but not difficult to verify that:
\begin{itemize}
  \item Each pair of distinct vertices from $V(G_{12}) = \{1,\ldots, 12\}$ belongs to either some $C\in {\cal C}$ or to some $S\in {\cal S}$ (but not both).
  \item Families ${\cal C}$ and ${\cal S}$ are the families of all maximal cliques and all maximal stable sets of $G_{12}$, respectively.
\end{itemize}

\begin{proposition}\label{prop:G12-weakly-CIS}
Graph $G_{12}$ is weakly CIS.
\end{proposition}

\begin{proof}
Let us consider the subfamilies
\begin{eqnarray*}
{\cal C}' &=&  \{\{1,4,7\},\{2,4,9,12\},\{2,5,8\},\{2,6,7,11\},\{3,6,9\},\{4,5,6,10\},\{10,11,12\}\}\,,\\
{\cal S}' &=&  \{\{1,2,3,10\},\{1,5,9,11\},\{1,6,8,12\},\{3,4,8,11\},\{3,5,7,12\},\{7,8,9,10\}\}
\end{eqnarray*}
of ${\cal C}$ and of ${\cal S}$, respectively.
It is not difficult to verify that ${\cal C}'$ and ${\cal S}'$ are cross-intersecting and cover all edges and non-edges of $G_{12}$, respectively.
This shows that $G_{12}$ is weakly CIS.
\qed
\end{proof}

\begin{proposition}
Graph $G_{12}$ is not $\cup$-semi-weakly CIS.
\end{proposition}

\begin{proof}
First, we show that $G_{12}$ is not semi-weakly CIS.
Indeed, the only maximal cliques that contain the edge $e = \{10,11\}$
are $\{6,10,11\}$ and $\{10,11,12\}$, neither of which is strong since both miss the maximal stable set $\{5,7,9\}$.
Thus, edge $e$ does not belong to any strong clique, and consequently $G_{12}$
does not admit an edge covering family of strong cliques.

Similarly, we show that $G_{12}$ is not co-semi-weakly CIS.
Indeed, the only maximal stable sets that contain the non-edge $\overline{e} = \{5,7\}$
are $\{5,7,9\}$ and $\{3,5,7,12\}$, neither of which is
strong,  since both miss the maximal clique $\{6,10,11\}$.
Thus, non-edge $\overline{e}$ does not belong to any strong stable set,
and consequently $G_{12}$ does not admit a non-edge covering family of strong stable sets.
\qed
\end{proof}

\begin{remark}
The above example was obtained from the finite projective plane of order $3$ by deleting one point from it; see, e.g.,~\cite{HP73}.
Every clique in ${\cal C}'$ of order $4$ and
every clique in ${\cal C}'$ of order $3$ together with the deleted point
form a line in the plane.
\end{remark}

\noindent The fact that every $\cup$-semi-weakly CIS graph is $\cup$-strongly equistable
follows from Theorem~\ref{thm:gpg-seq}. {The inclusion is proper, since
the graph $G_{22} = \overline{L(C_{11}(\{1,3\}))}$ is strongly equistable (and hence $\cup$-strongly equistable)
but it is neither semi-weakly CIS (see~\cite{MTT}) nor co-semi-weakly CIS, as we now verify by proving the following stronger claim.

\begin{proposition}\label{prop:G22}
The graph $G_{22} = \overline{L(C_{11}(\{1,3\}))}$ is not co-triangle.
\end{proposition}

\begin{proof}
Suppose for a contradiction that $G_{22} = \overline{L(C_{11}(\{1,3\}))}$ is co-triangle.
Then $L(C_{11}(\{1,3\}))$ is triangle. Since $G_{22}$ is equistable, it is also triangle
by Theorem~\ref{thm:triangle}, which implies that $L(C_{11}(\{1,3\}))$ is co-triangle.
Thus, $L(C_{11}(\{1,3\}))$ is $\cap$-triangle, which implies by Theorem~\ref{thm:line-equistable-co-equistable}
that $L(C_{11}(\{1,3\}))$ is CIS. This, however, is a contradiction, since any  maximal matching $M$ in
$C_{11}(\{1,3\})$
and the set of all edges incident with any vertex $v\in V(C_{11}(\{1,3\}))$ not covered by $M$
constitute a pair of a maximal stable set and a maximal clique that are disjoint.
\qed\end{proof}}

%
%
%

Regarding the remaining non-inclusions, we have:
\begin{itemize}
  \item Not every $\cup$-semi-weakly CIS graph is triangle (since not every $\cup$-edge simplicial graph is triangle).

  \item Not every $\cup$-semi-weakly CIS graph is quasi CIS (since not every $\cap$-semi-weakly CIS graph is quasi CIS).

  \item Not every $\cup$-semi-weakly CIS graph is perfect (since not every CIS graph is perfect).
\end{itemize}

\subsection{$\cap$-strongly equistable graphs}

The fact that every $\cap$-strongly equistable graph is $\cap$-equistable
follows from Theorem~\ref{thm:Seq-eq}. {We do not know whether the inclusion is proper or not.}
The fact that every $\cap$-strongly equistable graph is strongly equistable
follows from definition; the inclusion is proper
(e.g., $S_3$ is semi-weakly CIS hence strongly equistable but it is not co-triangle, hence also not
co-strongly equistable). {We do not know whether every $\cap$-strongly equistable graph is $\cup$-semi-weakly CIS,
weakly CIS, normal, or none of these. }
Regarding the remaining non-inclusions, we have:
\begin{itemize}
  \item Not every $\cap$-strongly equistable graph is quasi CIS (since not every $\cap$-semi-weakly CIS graph is quasi CIS).

  \item Not every $\cap$-strongly equistable graph is perfect (since not every CIS graph is perfect).
\end{itemize}

\subsection{Strongly equistable graphs}

The fact that every strongly equistable graph is equistable
is the statement of Theorem~\ref{thm:Seq-eq}; {the inclusion is proper
since the graph $G_{14}$, the complement of which is depicted in Fig.~\ref{fig:example}, is
equistable but not strongly equistable (see~\cite{MTT}).}
The fact that every strongly equistable graph is $\cup$-strongly equistable
follows from definition; the inclusion is proper
(e.g., the net is co-semi-weakly CIS and hence co-strongly equistable but it is not triangle, hence also not
strongly equistable). {We do not know whether every strongly equistable graph is weakly CIS,
$\cap$-weakly triangle, normal, or none of these. }
Regarding the remaining non-inclusions, we have:
\begin{itemize}
  \item Not every strongly equistable graph is $\cap$-triangle (since not every semi-weakly CIS graph is $\cap$-triangle).

  \item {Not every strongly equistable graph is $\cup$-semi-weakly CIS. A separating graph is given, for example, by
 $G_{22}$, which is strongly equistable but neither semi-weakly CIS (see~\cite{MTT}) nor co-semi-weakly CIS (see Proposition~\ref{prop:G22}). }

  \item Not every strongly equistable graph is quasi CIS (since not every $\cap$-semi-weakly CIS graph is quasi CIS).

  \item Not every strongly equistable graph is perfect (since not every CIS graph is perfect).
\end{itemize}

\subsection{$\cup$-strongly equistable graphs}

The fact that every $\cup$-strongly equistable graph is $\cup$-equistable
follows from Theorem~\ref{thm:Seq-eq}. {The inclusion is proper:
the graph $G_{14} = \overline{L(G^*)}$ (cf.~Fig.~\ref{fig:example}) is
equistable (and hence $\cup$-equistable) but it is neither
strongly equistable nor co-strongly equistable. The fact that it is not co-strongly equistable
is a consequence of the following stronger result, the proof of which is analogous to the proof of
Proposition~\ref{prop:G22}.}

\begin{proposition}\label{prop:G14}
The graph  $G_{14} = \overline{L(G^*)}$ is not co-triangle.
\end{proposition}

{We do not know whether every $\cup$-strongly equistable graph is weakly CIS,
$\cap$-weakly triangle, weakly triangle, normal, or none of these. }
Regarding the remaining non-inclusions, we have:
\begin{itemize}
  \item Not every $\cup$-strongly equistable graph is triangle. For example, the net is not.

  \item Not every $\cup$-strongly equistable graph is quasi CIS (since not every $\cap$-semi-weakly CIS graph is quasi CIS).

  \item Not every $\cup$-strongly equistable graph is perfect (since not every CIS graph is perfect).
\end{itemize}

\subsection{$\cap$-equistable graphs}

The fact that every $\cap$-equistable graph is $\cap$-triangle
follows from Theorem~\ref{thm:triangle}. To see that the inclusion is proper,
let $L$ denote
the graph obtained from the line graph of $K_{5,6}$ (see Fig.~\ref{fig:L}) by gluing a new triangle along every edge,
and consider the graph $L\overline{L}$, the disjoint union of graph $L$ and its complement.

\begin{figure}[h!]
\begin{center}
\includegraphics[width=0.25\textwidth]{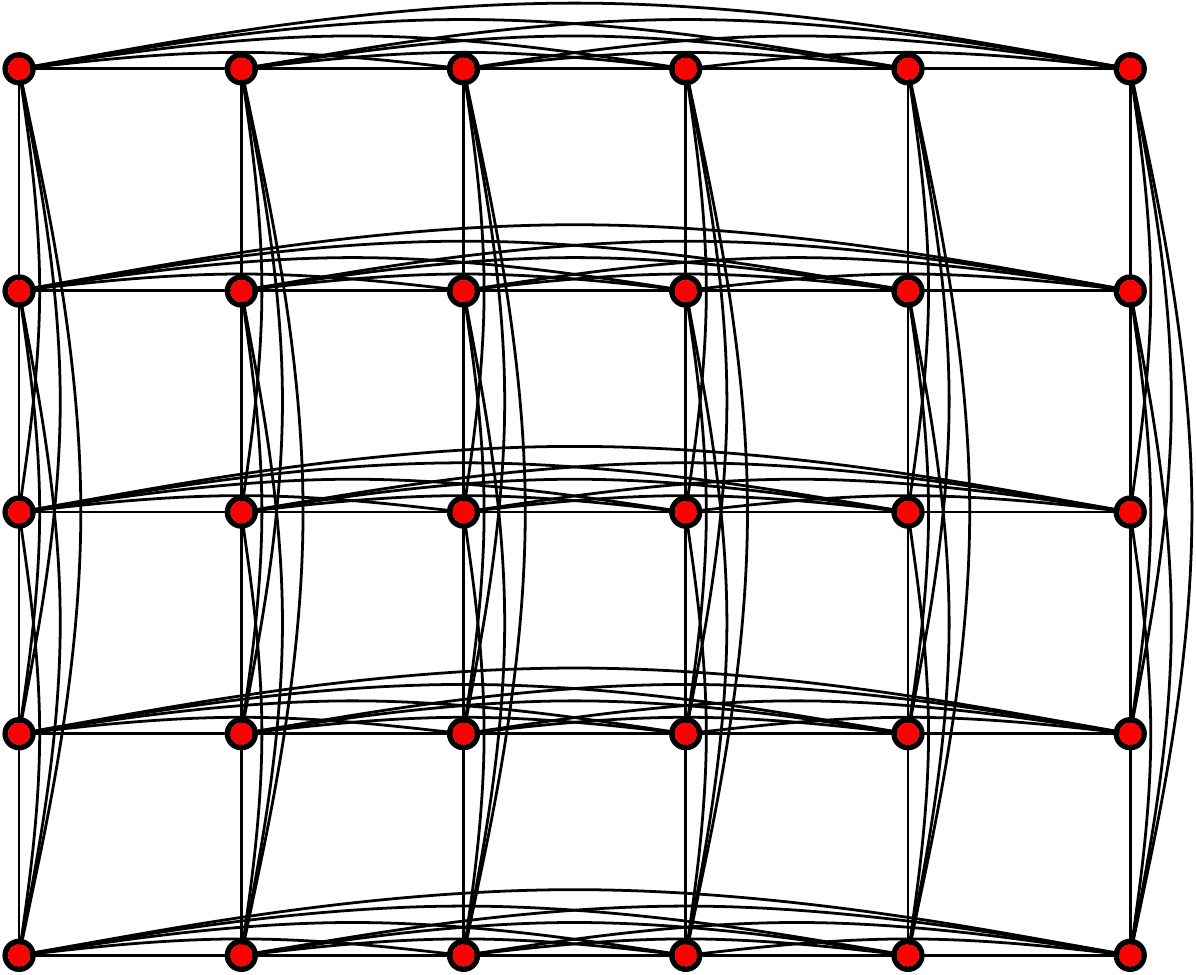}
\caption{The line graph of $K_{5,6}$.}
\label{fig:L}
\end{center}
\end{figure}

\begin{proposition}
Graph $L\overline{L}$ is $\cap$-triangle.
\end{proposition}

\begin{proof}
It is not difficult to see that graph $L$ is edge simplicial, and hence triangle.
Also, $L$ is co-triangle: for every maximal clique $C$ of $L$ and every non-edge $uv$ in $L-C$,
there exists a vertex in $C$ non-adjacent to both $u$ and $v$.
Indeed, either $|C| = 3$ (a triangle), in which case the vertex of degree $2$ in $C$ is non-adjacent to both $u$ and $v$,
or $|C| \in \{5,6\}$, in which case every vertex not in $C$ has at most $2$ neighbors in $C$,
which implies that there exists a vertex in $C$ non-adjacent to both $u$ and $v$.

Since both $L$ and $\overline{L}$ are triangle, so is their disjoint union $L\overline{L}$ (see Proposition~\ref{prop:triangle-disjoint-union-join}).
Similarly, Proposition~\ref{prop:triangle-disjoint-union-join} implies that the co-triangle graphs are closed under disjoint union.
Therefore, $L\overline{L}$ is also co-triangle, and consequently $\cap$-triangle.
\qed
\end{proof}

\noindent It remains to show that the graph $L\overline{L}$
is not $\cap$-equistable. Instead, we show the following stronger statement (that will also be useful later).

\begin{proposition}\label{prop:LLbar-not-cup-equistable}
Graph $L\overline{L}$ is not $\cup$-equistable.
\end{proposition}

\begin{proof}
We will show that $L\overline{L}$ is neither equistable nor co-equistable.
First we show that
$L\overline{L}$ is not co-equistable.
Suppose for a contradiction that
there exists a weight function $\varphi:V(L\overline{L})\to\R_+$
such that a subset $C\subseteq V(L\overline{L})$ is a maximal clique
of $L\overline{L}$ if and only if $\varphi(C)=1$.
The component $L$ of $L\overline{L}$ contains
$6$ pairwise disjoint maximal cliques $C_1,\ldots, C_6$ of order $5$,
and
$5$ pairwise disjoint maximal cliques $C^1,\ldots, C^5$ of order $6$.
Moreover, $\cup_{i = 1}^6C_i = \cup_{j = 1}^5C^j = V(L)$.
Clearly, these are also maximal cliques in $L\overline{L}$.
We now obtain a contradicting equality
$6 = \sum_{i = 1}^6\varphi(C_i) = \varphi(V(L)) = \sum_{j = 1}^5\varphi(C^j) = 5$.

The proof that $L\overline{L}$ is not equistable is only slightly more involved.
Suppose for a contradiction that
there exists a weight function $\varphi:V(L\overline{L})\to\R_+$
such that
\begin{equation}\label{eq3}
\text{ a subset $S\subseteq V(L\overline{L})$ is a maximal stable set
of $L\overline{L}$ if and only if $\varphi(S)=1$.}
\end{equation}
The component $\overline{L}$ of $L\overline{L}$ contains
$6$ pairwise disjoint maximal stable sets $S_1,\ldots, S_6$ of order $5$,
and
$5$ pairwise disjoint maximal stable sets $S^1,\ldots, S^5$ of order $6$.
Moreover, $\cup_{i = 1}^6S_i = \cup_{j = 1}^5S^j = V(\overline L)$.
Fix a maximal stable set $S$ of $L$. Then, each $S\cup S_i$ and each $S\cup S^j$ is a maximal stable set
in $L\overline{L}$.
Consequently, we have
$$6 = \sum_{i = 1}^6\varphi(S\cup S_i) = 6\varphi(S)+\varphi(V(L))\,,$$
and
$$5 = \sum_{j = 1}^5\varphi(S\cup S^j) = 5\varphi(S)+\varphi(V(L))\,.$$
This implies $\varphi(S) = 1$ and consequently
$\varphi(x) = 0$ for all $x\in V(\overline L)$. Therefore $\varphi(S\cup e) =1$ where $e$ is an edge in $\overline L$,
a contradiction with~\eqref{eq3}.\qed
\end{proof}

\noindent The fact that every $\cap$-equistable graph is equistable
follows from definition; the inclusion is proper (e.g., the net is co-edge-simplicial and hence co-equistable but not triangle, hence also not equistable).
{We do not know whether every $\cap$-equistable graph is $\cup$-semi-weakly CIS,
$\cup$-strongly equistable, weakly CIS, normal, or none of these. }
Regarding the remaining non-inclusions, we have:
\begin{itemize}
  \item Not every $\cap$-equistable graph is quasi CIS (since not every $\cap$-semi-weakly CIS graph is quasi CIS).

  \item Not every $\cap$-equistable graph is perfect (since not every CIS graph is perfect).
\end{itemize}

\subsection{Equistable graphs}

The fact that every equistable graph is triangle is the statement of Theorem~\ref{thm:triangle}; the inclusion is proper
(e.g., the $9$-vertex graph $\textrm{Cir}_9$ is triangle but not equistable, see~\cite{MM-2011}).
The fact that every equistable graph is $\cup$-equistable
follows from definition; the inclusion is proper
(e.g., the net is co-semi-weakly CIS and hence co-equistable and $\cup$-equistable, but it is not triangle, hence also not equistable).
\\
{We do not know whether every equistable graph is weakly CIS,
$\cap$-weakly triangle, normal, or none of these. }
Regarding the remaining non-inclusions, we have:
\begin{itemize}
\item Not every equistable graph is $\cup$-strongly equistable.
{A separating graph is given, for example, by $G_{14}$; this graph is equistable but not strongly equistable, and by Proposition~\ref{prop:G14}
also not co-triangle (hence, in particular, not co-strongly equistable).}

\item Not every equistable graph is $\cap$-triangle (since not every semi-weakly CIS graph is $\cap$-triangle).

  \item Not every equistable graph is quasi CIS (since not every $\cap$-semi-weakly CIS graph is quasi CIS).

  \item Not every equistable graph is perfect (since not every CIS graph is perfect).
\end{itemize}

\subsection{$\cup$-equistable graphs}

The fact that every $\cup$-equistable graph is $\cup$-triangle
follows from Theorem~\ref{thm:triangle}.
The inclusion is proper; a separating graph can be obtained by the $9$-vertex circle graph $\textrm{Cir}_9$ is $\cup$-triangle but not $\cup$-equistable.
Recall that $V(\textrm{Cir}_9)=\{1,\ldots,9\}$ and ${\cal S}(\textrm{Cir}_9) = \{\{1,2,3\},\{4,5,6\},\{7,8,9\},\{1,4,7\},\{3,6,9\}\}$.

\begin{proposition}\label{prop:Cir9-not-cup-equistable}
Graph $\textrm{Cir}_9$ is triangle but not $\cup$-equistable.
\end{proposition}

\begin{proof}
It is known that $\textrm{Cir}_9$ is triangle~\cite{DTDRM}
but not equistable~\cite{MM-2011}.
It can be easily seen that $\textrm{Cir}_9$ is also not co-equistable.
Indeed, let $C_1 = \{1,5,9\}$, $C_2 = \{2,6,7\}$, $C_3 = \{3,4,8\}$, $C_4 = \{1,6,8\}$, $C_5 = \{2,5,7\}$, and
let $x = x(C_1)+x(C_2)+x(C_3)-x(C_4)-x(C_5)$, where $x(C)$ denotes the
characteristic vector of a set $C$. Then, each $C_i$ is a maximal clique in $\textrm{Cir}_9$,
and $x\in\{0,1\}^{V(\textrm{Cir}_9)}$ is the characteristic vector of the set $T =\{3,4,9\}$, which
is not a clique in $\textrm{Cir}_9$.
Suppose that $\varphi:V\to \R^+$ is an equistable weight function of the complement of $\textrm{Cir}_9$.
Writing $\phi = (\varphi(v)\,:\,v\in V(\textrm{Cir}_9))$, it follows that
$$\varphi(T) = \phi^\top x = \phi^{\top}\left(x(C_1)+x(C_2)+x(C_3)-x(C_4)-x(C_5)\right)
= 1+1+1-1-1=1\,,$$ a contradiction.
\qed\end{proof}
\\
{We do not know whether every $\cup$-equistable graph is weakly CIS,
$\cap$-weakly triangle, weakly triangle, normal, or none of these. }
Regarding the remaining non-inclusions, we have:
\begin{itemize}
  \item Not every $\cup$-equistable graph is quasi CIS (since not every $\cap$-semi-weakly CIS graph is quasi CIS).

  \item Not every $\cup$-equistable graph is perfect (since not every CIS graph is perfect).
\end{itemize}

\subsection{$\cap$-triangle graphs}

The fact that every $\cap$-triangle graph is triangle
follows from definition; the inclusion is proper
(e.g., $S_3$ is triangle but not co-triangle).
The fact that every
$\cap$-triangle graph is $\cap$-weakly triangle
follows from the fact that every triangle graph is weakly triangle (cf.~the discussion following Definition~\ref{def:wt}
on p.~\pageref{triangle-weakly-triangle}). The inclusion is proper: a separating graph is, for example, $G_{12}$ (recall the definition of this graph in
Section~\ref{sec:G12}). By Proposition~\ref{prop:G12-weakly-CIS}, $G_{12}$ is weakly CIS, and hence also $\cap$-weakly triangle,
by Proposition~\ref{prop:weakly-CIS-cap-weakly-triangle}.
It remains to show that $G_{12}$ is not $\cap$-triangle. This will follow from the following stronger claim.

\begin{proposition}\label{prop:G12-not-cup-triangle}
Graph $G_{12}$ is not $\cup$-triangle.
\end{proposition}

\begin{proof}
Consider the maximal stable set $S = \{5,7,9\}$ and the edge $e = \{10,11\}$.
The only maximal cliques containing vertices $10$ and $11$ are $\{6,10,11\}$ and $\{10,11,12\}$,
therefore the endpoints of $e$ do not have a common neighbor in $S$, which shows that $G_{12}$ is not triangle.

Similarly, consider the maximal clique $C = \{4,10,12\}$ and the non-edge $\overline{e} = \{5,9\}$.
The only maximal stable sets containing $5$ and $9$ are $\{1,5,9,11\}$ and $\{5,7,9\}$,
showing that $5$ and $9$ do not have a common non-neighbor in $C$. Thus, $G_{12}$ is also not co-triangle, which completes the proof.\qed
\end{proof}
\\
We do not know whether every $\cap$-triangle graph is weakly CIS, or whether
every $\cap$-triangle graph is normal.
Regarding the remaining non-inclusions, we have:
\begin{itemize}
  \item Not every $\cap$-triangle graph is quasi CIS (since not every $\cap$-semi-weakly CIS graph is quasi CIS).

  \item Not every $\cap$-triangle graph is perfect (since not every CIS graph is perfect).
\end{itemize}

\subsection{Triangle graphs}\label{sec:triangle}

The fact that every triangle graph is $\cup$-triangle
follows directly from the definition; the inclusion is proper
(e.g., the net is co-edge-simplicial and hence co-triangle but not triangle).
The fact that every triangle graph is weakly triangle follows immediately from the definition
(cf.~the discussion following Definition~\ref{def:wt} on p.~\pageref{triangle-weakly-triangle});
the inclusion is proper (e.g., graph $G_{12}$ is weakly CIS (see Proposition~\ref{prop:G12-weakly-CIS}), hence also
weakly triangle, but it is not triangle, see Proposition~\ref{prop:G12-not-cup-triangle}).
\\
We do not know whether every triangle graph is normal.
Regarding the remaining non-inclusions, we have:
\begin{itemize}
  \item Not every triangle graph is $\cap$-weakly triangle. For example, $\textrm{Cir}_9$ is not (this follows from Proposition~\ref{prop:Cir9-not-co-weakly-triangle} below.

  \item Not every triangle graph is quasi CIS (since not every $\cap$-semi-weakly CIS graph is quasi CIS).
  \item Not every triangle graph is perfect (since not every CIS graph is perfect).
\end{itemize}

\begin{proposition}\label{prop:Cir9-not-co-weakly-triangle}
Graph $\textrm{Cir}_9$ is not co-weakly triangle.
\end{proposition}

\begin{proof}
Let ${\cal C}$ be an edge-covering collection of maximal cliques of $\textrm{Cir}_9$, and let
$C\in {\cal C}$ be a clique in ${\cal C}$ such that $\{2,5\}\subseteq C$.
Since $7$ and $9$ are non-adjacent, we have $\{7,9\}\nsubseteq C$, and
we may assume without loss of generality that $7\not\in C$.
Since $4$ is non-adjacent to $5$, we also have $4\not\in C$. Therefore, $47$ is a non-edge in $\textrm{Cir}_9-C$.
It can be seen that $4$ and $7$ do not have a common neighbor in $C$. Since the choice of ${\cal C}$ was arbitrary,
this implies that $\textrm{Cir}_9$ is not co-weakly triangle.\qed
\end{proof}

\subsection{$\cup$-triangle graphs}

The fact that every $\cup$-triangle graph is $\cup$-weakly triangle
follows from the fact that every triangle graph is weakly triangle; the inclusion is proper
(e.g., graph $G_{12}$ is weakly triangle, but not $\cup$-triangle, see Proposition~\ref{prop:G12-not-cup-triangle}).
\\
We do not know whether every $\cup$-triangle graph is normal.
Regarding the remaining non-inclusions, we have:
\begin{itemize}
  \item Not every $\cup$-triangle graph is weakly triangle. For example, the complement $\textrm{Cir}_9$ is not
  (it follows from Proposition~\ref{prop:Cir9-not-co-weakly-triangle} that the complement of $\textrm{Cir}_9$ is
  not weakly triangle; however it is co-triangle since $\textrm{Cir}_9$ is triangle).
  \item Not every $\cup$-triangle graph is quasi CIS (since not every $\cap$-semi-weakly CIS graph is quasi CIS).
  \item Not every $\cup$-triangle graph is perfect (since not every CIS graph is perfect).
\end{itemize}

\subsection{Weakly CIS graphs}

The fact that every weakly CIS graph is $\cap$-weakly triangle
is the statement of Proposition~\ref{prop:weakly-CIS-cap-weakly-triangle};
we do not know whether the inclusion is proper.
The fact that every weakly CIS graph is normal
is the statement of Proposition~\ref{prop:weakly-CIS-normal}; the inclusion is proper (e.g., $P_4$ is normal but not weakly CIS, as can be easily verified).
\\
Regarding the remaining non-inclusions, we have:
\begin{itemize}
  \item Not every weakly CIS graph is quasi CIS (since not every $\cap$-semi-weakly CIS graph is quasi CIS).
  \item Not every weakly CIS graph is perfect (since not every CIS graph is perfect).
\end{itemize}

\subsection{$\cap$-weakly triangle graphs}

The fact that every $\cap$-weakly triangle graph is weakly triangle follows from definition; the inclusion is proper
(e.g., $\textrm{Cir}_9$ is weakly triangle but not co-weakly triangle).
\\
We do not know whether every $\cap$-weakly triangle graph is weakly CIS, or whether
every $\cap$-weakly triangle graph is normal.
Regarding the remaining non-inclusions, we have:
\begin{itemize}
  \item Not every $\cap$-weakly triangle graph is quasi CIS (since not every $\cap$-semi-weakly CIS graph is quasi CIS).

  \item Not every $\cap$-weakly triangle graph is perfect (since not every CIS graph is perfect).
\end{itemize}

\subsection{Weakly triangle graphs}

The fact that every weakly triangle graph is $\cup$-weakly triangle follows from definition; the inclusion is proper
(e.g., the complement of $\textrm{Cir}_9$ is co-triangle and hence co-weakly triangle but not weakly triangle).
\\
We do not know whether every weakly triangle graph is normal.
Regarding the remaining non-inclusions, we have:
\begin{itemize}
  \item Not every weakly triangle graph is quasi CIS (since not every $\cap$-semi-weakly CIS graph is quasi CIS).

  \item Not every weakly triangle graph is perfect (since not every CIS graph is perfect).
\end{itemize}

\subsection{$\cup$-weakly triangle graphs}

The above examples imply that the class of $\cup$-weakly triangle graphs is not contained in any of the other graph classes considered in Fig.~\ref{fig:Hasse},
except possibly in the class of normal graphs.

\subsection{Quasi CIS graphs}

The fact that every quasi CIS graph is normal follows from the following facts:
(1) by definition, every quasi CIS graph is either CIS or almost CIS, (2)
every CIS graph is normal (this is a consequence of Propositions~\ref{prop:CIS-sw-CIS-w-CIS}
and~\ref{prop:weakly-CIS-normal}), and (3) every almost CIS graph is normal (Corollary~\ref{cor:almost-CIS-normal}).
The inclusion is proper (e.g., ${\it FK}$ is normal but not quasi CIS).
\\
Regarding the remaining non-inclusions, we have:
\begin{itemize}
  \item Not every quasi CIS graph is $\cup$-weakly triangle (since not every almost CIS graph is $\cup$-weakly triangle).
  \item Not every quasi CIS graph is perfect (since not every CIS graph is perfect).
\end{itemize}

\subsection{Perfect graphs}

The fact that every perfect graph is normal is the statement of Proposition~\ref{prop:perfect-normal};
the inclusion is proper (e.g., ${\it C_9}$ is normal~\cite{DeSimoneKoerner99} but not perfect).
\\
Regarding the remaining non-inclusions, we have:
\begin{itemize}
  \item Not every perfect graph is $\cup$-weakly triangle (since not every almost CIS graph is $\cup$-weakly triangle).
\end{itemize}

\subsection{Normal graphs}

Normal graphs are not contained in any of the other graph classes considered in Fig.~\ref{fig:Hasse}.
To see this, it is sufficient to observe that not every normal graph is weakly CIS, quasi CIS, or perfect (as justified above), and
that not every normal graph is $\cup$-weakly triangle (since not every almost CIS graph is $\cup$-weakly triangle).

\subsection{Summary of inclusion relations and separating examples for some of the classes}

We conclude this section by summarizing in Table~\ref{fig:table}
below the known inclusion relations between the following $17$ self-complementary graph properties:
CIS, almost CIS, quasi CIS, split,
$\cap$-edge simplicial, $\cup$-edge simplicial,
$\cap$-equistable, $\cup$-equistable,
$\cap$-strongly equistable, $\cup$-strongly equistable,
$\cap$-triangle, $\cup$-triangle,
$\cap$-weakly triangle, $\cup$-weakly triangle,
$\cap$-semi-weakly CIS, $\cup$-semi-weakly CIS,
and weakly CIS graphs.
Symbol $\subseteq$ in row $X$, column $Y$ of Table~\ref{fig:table} means that $X\subseteq Y$.
If $X\nsubseteq Y$, then the element in row $X$, column $Y$ is an {\it $(X,Y)$-separating graph}, that is, a graph belonging to $X\setminus Y$.
The abbreviations for classes are as follows:
es (edge simplicial), gp (general partition), seq (strongly equistable), eq (equistable),
aCIS (almost CIS), wCIS (weakly CIS), $\triangle$ (triangle), w$\triangle$ (weakly triangle).

\begin{table}[h!]
\centering
{
\footnotesize
\renewcommand{\arraystretch}{1.2}
\tabcolsep=0.1cm
%
%
\begin{tabular}{|c||c|c|c|c|c|c|c|c|c|c|c|c|c|c|c|c|c|}
  \hline
  $X \backslash Y$ &  aCIS& $\cap$-es & split & CIS & qCIS & $\cap$-swCIS & wCIS & $\cap$-seq & $\cap$-eq &
  $\cap$-$\triangle$ & $\cap$-\textrm{w}$\triangle$ & $\cup$-\textrm{es} & $\cup$-swCIS & $\cup$-\textrm{seq} & $\cup$-\textrm{eq} & $\cup$-$\triangle$ & $\cup$-\textrm{w}$\triangle$\\
  \hline
  \hline
aCIS & = & $P_4$ & $\subseteq$ & $P_4$ & $\subseteq$ & $P_4$ & $P_4$ & $P_4$ & $P_4$ & $P_4$ & $P_4$ &
$P_4$ & $P_4$ & $P_4$ & $P_4$ & $P_4$ & $P_4$ \\\hline

$\cap$-\textrm{es} & $K_1$ & $=$ & $\subseteq$ & $F$ & $\subseteq$ & $\subseteq$ &
$\subseteq$ & $\subseteq$ & $\subseteq$ & $\subseteq$ & $\subseteq$ & $\subseteq$ &
$\subseteq$ & $\subseteq$ & $\subseteq$ & $\subseteq$ & $\subseteq$ \\\hline

split & $K_1$ & $P_4$ & $=$ & $P_4$ & $\subseteq$ & $P_4$
& $P_4$ & $P_4$ & $P_4$ & $P_4$ & $P_4$ &
$P_4$ & $P_4$ & $P_4$ & $P_4$ & $P_4$ & $P_4$ \\\hline

CIS & $K_1$ & $C_4$ & $C_4$ & $=$ & $\subseteq$ & $\subseteq$ &
$\subseteq$ & $\subseteq$ & $\subseteq$ & $\subseteq$ & $\subseteq$ &
${\it CK}$
&
$\subseteq$ & $\subseteq$ & $\subseteq$ & $\subseteq$ & $\subseteq$ \\\hline

qCIS & $K_1$ & $C_4$ & $C_4$ & $P_4$ & $=$ & $P_4$ &
 $P_4$ & $P_4$ & $P_4$ & $P_4$ & $P_4$ & $P_4$  &
$P_4$ & $P_4$ & $P_4$ & $P_4$ & $P_4$ \\\hline

$\cap$-swCIS &
$K_1$ & $C_4$ & $C_4$ & $F$ &
${\it FK}$
 & $=$ & $\subseteq$ & $\subseteq$ & $\subseteq$ & $\subseteq$ & $\subseteq$ &
 ${\it CK}$ &
$\subseteq$ & $\subseteq$ & $\subseteq$ & $\subseteq$ & $\subseteq$ \\\hline

wCIS &
$K_1$ & $C_4$ & $C_4$ & $F$ &
$G_{12}$ & $G_{12}$ &
$=$ & {$G_{12}$} & {$G_{12}$} & {$G_{12}$}
& $\subseteq$ & $L(K_{3,3})$ & $G_{12}$ & {$G_{12}$} &{$G_{12}$} &
{$G_{12}$} & $\subseteq$\\\hline

$\cap$-seq &
$K_1$ & $C_4$ & $C_4$ & $F$ &
${\it FL}$
& $G_{12}$ & {\bf ?} &
$=$
& $\subseteq$
& $\subseteq$
& $\subseteq$
& $L(K_{3,3})$
& {\bf ?}
& $\subseteq$
& $\subseteq$
& $\subseteq$
& $\subseteq$\\\hline

$\cap$-eq &
$K_1$ & $C_4$ & $C_4$ & $F$ & ${\it FL}$
& $G_{12}$ & {\bf ?}
& {\bf ?} & $=$
& $\subseteq$
& $\subseteq$
& $L(K_{3,3})$
& {\bf ?}
& {{\bf ?}}
& $\subseteq$
& $\subseteq$
& $\subseteq$
\\\hline

$\cap$-$\triangle$ &
$K_1$ &
 $C_4$ &
 $C_4$ &
 $F$ &
 ${\it FL}$
&
$L\overline{L}$&
{\bf ?}
&
$L\overline{L}$&
$L\overline{L}$&
 $=$
& $\subseteq$
& $L(K_{3,3})$
& $L\overline{L}$
& $L\overline{L}$
& $L\overline{L}$
& $\subseteq$
& $\subseteq$
\\\hline

$\cap$-w$\triangle$ &
$K_1$ &
 $C_4$ &
 $C_4$ &
 $F$ &
 ${\it FL}$
&
$L\overline{L}$&
{\bf ?}
& $L\overline{L}$
& $L\overline{L}$
& $G_{12}$
& $=$
& $L(K_{3,3})$
& $L\overline{L}$
& $L\overline{L}$
& $L\overline{L}$
& $G_{12}$
& $\subseteq$
\\\hline

$\cup$-es &
$K_1$ &
$C_4$ &
$C_4$ &
$S_3$ &
${\it SK}$ &
$S_3$ &
$\subseteq$ &
$S_3$ &
$S_3$ &
$S_3$ &
$\subseteq$
& $=$
& $\subseteq$
& $\subseteq$
& $\subseteq$
& $\subseteq$
& $\subseteq$
\\\hline

$\cup$-swCIS &
$K_1$ &
$C_4$ &
$C_4$ &
$S_3$ &
${\it SK}$ &
$S_3$ &
$\subseteq$ &
$S_3$ &
$S_3$ &
$S_3$ &
$\subseteq$
& $L(K_{3,3,})$
& $=$
& $\subseteq$
& $\subseteq$
& $\subseteq$
& $\subseteq$
\\\hline

$\cup$-seq &
$K_1$ &
$C_4$ &
$C_4$ &
$S_3$ &
${\it SK}$ &
$S_3$ &
{\bf ?}
&
$S_3$ &
$S_3$ &
$S_3$ &
{\bf ?}
& $L(K_{3,3})$
& {$G_{22}$}
& $=$
& $\subseteq$
& $\subseteq$
& $\subseteq$
\\\hline

{$\cup$-eq} &
$K_1$ &
$C_4$ &
$C_4$ &
$S_3$ &
${\it SK}$ &
$S_3$ &
{\bf ?}
&
$S_3$ &
$S_3$ &
$S_3$ &
{\bf ?}
& $L(K_{3,3})$
& {$G_{22}$}
& {$G_{14}$}
& $=$
& $\subseteq$
& $\subseteq$
\\\hline

$\cup$-$\triangle$ &
$K_1$ &
$C_4$ &
$C_4$ &
$S_3$ &
${\it SK}$ &
$S_3$ &
$\textrm{Cir}_9$ &
$S_3$ &
$S_3$ &
$S_3$ &
$\textrm{Cir}_9$ &
$L(K_{3,3})$ &
$\textrm{Cir}_9$ &
$\textrm{Cir}_9$ &
$\textrm{Cir}_9$ &
 $=$
& $\subseteq$
\\\hline

$\cup$-w$\triangle$ &
$K_1$ &
$C_4$ &
$C_4$ &
$S_3$ &
${\it SK}$ &
$S_3$ &
$\textrm{Cir}_9$ &
$S_3$ &
$S_3$ &
$S_3$ &
$\textrm{Cir}_9$ &
$L(K_{3,3})$ &
$\textrm{Cir}_9$ &
$\textrm{Cir}_9$ &
$\textrm{Cir}_9$ &
$G_{12}$
&
 $=$
\\\hline
\end{tabular}
}
\caption{Known inclusion relations between $17$ self-complementary graph properties. Symbol $\subseteq$ in row $X$, column $Y$ means that $X\subseteq Y$.
If $X\nsubseteq Y$, the element in row $X$, column $Y$ is a graph belonging to $X\setminus Y$.
A question mark means that it is not known (to us) whether $X\subseteq Y$.}
\label{fig:table}
\end{table}

\section{Line graphs}\label{sec:line}

In this section, we show that a part of the Hasse diagram in Fig.~\ref{fig:Hasse} simplifies greatly if we restrict ourselves to line graphs.
It was shown in~\cite{LM} that for line graphs, the triangle property implies the general partition (equivalently, the
$\cap$-semi-weakly CIS) property.
This implies that for line graphs, the $\cap$-triangle property implies the $\cap$-semi-weakly CIS property.
Here, we strengthen this result by proving that for line graphs,
the $\cap$-triangle property implies the CIS property. We also give a polynomially testable characterization of
CIS line graphs.

\begin{theorem}\label{thm:line-equistable-co-equistable}
For every line graph $G=L(H)$, the following conditions are equivalent:
\begin{enumerate}
\item[$(i)$] $G$ is a CIS graph.
\item[$(ii)$] $G$ is a $\cap$-general partition graph.
\item[$(iii)$] $G$ is a $\cap$-semi-weakly CIS graph.
\item[$(iv)$] $G$ is a $\cap$-strongly equistable graph.
\item[$(v)$] $G$ is a $\cap$-equistable graph.
\item[$(vi)$] $G$ is a $\cap$-triangle graph.
\item [$(vii)$] No subgraph of $H$ is isomorphic to a bull, and for every maximal matching $M$ of $H$ and every vertex $x\in V(H)$ not covered by $M$, the neighborhood of $x$ is contained in some edge of $M$.
\item [$(viii)$] No subgraph of $H$ is isomorphic to a bull, and for every vertex $x\in V(H)$,
   every matching in $H(x)$ with at least two edges misses some neighbor of $x$. Here, $H(x)$ denotes the subgraph of $H$ induced by
 all edges of $H$ incident with a neighbor of $x$ but not with $x$.
\end{enumerate}
In particular, CIS line graphs can be recognized in polynomial time.
\end{theorem}

\begin{proof}
The chain of implications $(i)\Rightarrow (ii)\Rightarrow (iii) \Rightarrow (iv) \Rightarrow (v) \Rightarrow (vi)$ holds for arbitrary graphs (cf.~Section~\ref{sec:relations}).

We will show that $(vi)$ implies $(vii)$, that $(vii)$ implies $(i)$, and that $(vii)$ and $(viii)$ are equivalent.

First, we show the implication $(vi)\Rightarrow (vii)$.
Suppose that $G = L(H)$ and $G$ and $\overline{G}$ are triangle.
Since $\overline{G}$ is triangle, for every maximal clique $C$ of $G$ and
every pair of non-adjacent vertices $u,v\in V(G)\setminus C$, $u$ and $v$ have a common non-neighbor in $C$. In particular, this implies that
for every triangle $T$ in $H$ and every pair of disjoint edges $e,f\in E(H)\setminus E(T)$, there exists an edge of $T$ disjoint from $e$ and $f$.
If $H$ has a subgraph isomorphic to a bull, then the triangle $T$ of the bull together with the two edges $e$ and $f$ of the bull not in $T$ contradict the above property. Hence, no subgraph of $H$ is isomorphic to a bull. Suppose for a contradiction that there exists a maximal matching $M$ of $H$ and a vertex $x\in V(H)$ not covered by $M$ such that the neighborhood of $x$ is not contained in any edge of $M$. By maximality of $M$, every neighbor of $x$ is covered by an edge of $M$, and since
the neighborhood of $x$ is not contained in any edge of $M$, there exist two distinct edges $e$ and $f$ of $M$ such that $e\cap N_H(x)\neq\emptyset$ and $f\cap N_H(x)\neq\emptyset$. Choose $u\in e\cap N(x)$ and $v\in f\cap N(x)$ arbitrarily. Then $u\neq v$.
The set $M$ forms a maximal stable set in $G$ not containing any endpoint of the edge $\{ux,vx\}$. By the triangle condition, the vertices
$ux$ and $vx$ in $G$ have a common neighbor in $M$. But this is impossible, since, on the one hand, vertex $x$ is not covered by $M$, and, on the other hand,
$uv\not\in M$ since $M$ is a matching. Hence, for every maximal matching $M$ of $H$ and every vertex $x$ not covered by $M$, the neighborhood of $x$ is contained in some edge of $M$.

Next, we show the implication $(vii)\Rightarrow (i)$.
Suppose that no subgraph of $H$ is isomorphic to a bull, and for every maximal matching $M$ of $H$ and every vertex $x\in V(H)$ not covered by $M$, the neighborhood of $x$ is contained in some edge of $M$.
Suppose for a contradiction that $G$ is not CIS, that is, that there exists a maximal stable set $S$ in $G$ and a maximal clique $C$ in $G$ such that $C\cap S = \emptyset$. Then, there exists a maximal matching $M$ in $H$ that misses a set of edges $T$ where $T$ is either a triangle or an inclusion-wise maximal star not contained in any triangle. If $T$ is a triangle, then, by the maximality of $M$, matching $M$ must cover at least two vertices of a triangle. But this is impossible by the bull-freeness of $H$. Hence, $M$ misses a maximal star $T$ not contained in any triangle.
Let $x\in V(H)$ be a center of this star. By the assumption on $H$, the neighborhood of $x$ is contained in some edge of $M$. However, this means that either $x$ is a leaf, in which case $T$ is not maximal (since the unique neighbor of $x$ is covered by $M$), or $x$ is a simplicial vertex of degree $2$. But in this case, $T$ is contained in a triangle, which is a contradiction.

Next, we show the implication $(vii)\Rightarrow (viii)$.
Suppose that condition $(viii)$ fails. Then, there exists a vertex $x\in V(H)$ such that there exists a matching $M$ in $H(x)$ with at least two edges that covers all neighbors of $x$. Then, $M$ is a matching in $H$. Extend $M$ to a maximal matching $M'$. Since $M$ covers all neighbors of $x$ but not $x$ itself, $x$ is also not covered by $M'$. Moreover, since $|M|\ge 2$, we see that $x$ is a vertex such that its neighborhood $N_H(x)$ is not contained in any edge of $M'$. Thus, condition $(vii)$ is violated as well.

Finally, we show the implication $(viii)\Rightarrow (vii)$.
Suppose that $(vii)$ fails, that is, there exists a maximal matching $M$ of $H$ and a vertex $x$ not covered by $M$ such that the neighborhood of $x$ is not contained in any edge of $M$. Since $x$ is not covered by $M$ and $M$ is maximal, $M$ covers all neighbors of $x$.
Let $M':= M\cap E(H(x))$. Then, $M'$ is a matching in $H(x)$ covering all neighbors of $x$.
Since $N_H(x)$ is not contained in any edge of $M$, it follows that $|M'|\ge 2$. Thus, we have a vertex $x$ of $H$ such that
the graph $H(x)$ contains a matching with at least two edges that covers all neighbors of $x$.
Hence, condition  $(viii)$ is violated.

It remains to show that the equivalent conditions can be verified in polynomial time.
We do this by describing a polynomial time algorithm to verify property~$(viii)$.
Given a graph $G$, testing whether $G$ is a line graph and computing a root graph
$H$  of it (if one exists) can be done in linear time~\cite{Roussopoulos}.
Clearly, verifying whether $H$ has a subgraph isomorphic to a bull can be tested in polynomial time.

It remains to verify whether for every vertex $x\in V(H)$, every matching in $H(x)$ with at least two edges misses some neighbor of $x$.
This condition is satisfied vacuously whenever the matching number of $H(x)$ is at most $1$, which is clearly the case if $x$ is of degree at most $1$, and also, by the bull-freeness, if $x$ is a simplicial vertex of degree~$2$.

So suppose that $x$ is a vertex of degree at least $3$ or a non-simplicial vertex of degree $2$. We compute the graph $H(x)$, and solve an instance of the maximum weight matching problem on $H(x)$ where every edge in $H(x)$ connecting two vertices in $N_H(x)$ is assigned weight $2$, and every other edge in $H(x)$ gets weight $1$. It is easy to see that graph $H(x)$ has a matching of total weight $d_H(x)$ if and only if $H(x)$ contains a matching of at least two edges covering all vertices in $N_H(x)$ (the condition that at least two edges are in the matching is without loss of generality, since by assumption on $x$, it is not possible to cover all neighbors of $x$ with a single edge). Equivalently, every matching in $H(x)$ with at least two edges misses some neighbor of $x$
if and only if the maximum weight of a matching in $H(x)$ is of value less than $d_H(x)$.
Since the maximum weight matching problem is solvable in polynomial time (see, e.g.,~\cite{GT91}), and
all the graphs $H(x)$ together with their edge weights can also be computed in polynomial time, the conclusion follows.
\qed
\end{proof}

\medskip

Theorem~\ref{thm:line-equistable-co-equistable} provides an infinite family of CIS line graphs.
In particular, as a special case we obtain all graphs of the form $L(\tilde H)$ where $H$ is a triangle-free graph,
and $\tilde H$ is a graph obtained from $H$ by adding a new private neighbor to every vertex of $H$.
Formally:
    \begin{itemize}
      \item $V(L(\tilde H)) = E(H)\cup \{\{v,v'\}\wt v\in V(H)\}$ and
      \item $E(L(\tilde H)) = \{\{e,f\}\wt e,f\in V(L(\tilde H)) \wedge e\neq f\wedge e\cap f\neq\emptyset\}$.
    \end{itemize}
Theorem~\ref{thm:line-equistable-co-equistable} applies, since condition $(vii)$ can be easily verified to hold for $\tilde H$.

\section{Open questions}\label{sec:open}

Some of the inclusion relations in Fig.~\ref{fig:Hasse} remain undetermined (e.g., those represented by dotted lines).
We restate some of them here explicitly:
\begin{itemize}
 \item Is every $\cap$-triangle (or even $\cap$-weakly triangle) graph weakly CIS?

 \item Is the following weaker form of Orlin's conjecture true? Every equistable graph is weakly CIS.

 \item Is Orlin's conjecture true when restricted to $\cap$-equistable graphs?
 That is, is every $\cap$-equistable graph a $\cap$-general partition graph?

 \item Is every $\cap$-triangle (or even triangle or $\cup$-triangle) graph normal?

 \item Is every $\cap$-weakly triangle (or even weakly triangle or $\cup$-weakly triangle) graph normal?
\end{itemize}

There are also many open questions related to the algorithmic complexity of recognizing graph classes considered in this paper.
If $\Pi$ denotes one of the following properties: threshold, $\cap$-edge simplicial, edge simplicial, $\cup$-edge simplicial, cographs, almost CIS, split, perfect,
then it is known that recognizing graphs in $\Pi$ can be done in polynomial time.
To the best of our knowledge, the computational complexity status of recognizing each of the remaining $19$
classes depicted in Fig.~\ref{fig:Hasse} is open. Recognizing CIS graphs was conjectured to be polynomial by Andrade {et al.}~\cite{ABG06}, and co-NP-complete
  by Zverovich-Zverovich~\cite{ZZ} (who named CIS graphs {\it stable} graphs).
  Recognizing triangle graphs was conjectured to be co-NP-complete
  by Kloks {et al.}~\cite{KLLM}.


\end{document}